\documentclass[12pt]{amsart}

\usepackage[hmargin=2cm,bottom=3cm,textwidth=17.3cm]{geometry}

\usepackage{ytableau}
\usepackage[latin1]{inputenc}
\usepackage[all]{xy}
\usepackage{mathrsfs}
\usepackage{amsmath}
\usepackage{amssymb}
\usepackage{amsthm}
\usepackage{pdfsync}
\usepackage[percent]{overpic}
\usepackage[linktocpage=true]{hyperref}
\usepackage{fancyhdr}

\newcommand{\op}{\operatorname}

\newcommand{\stb}{,\ldots,}

\newcommand{\coker}{\operatorname{coker}}

\newcommand{\rk}{\operatorname{rk}}
\newcommand{\Hom}{\operatorname{Hom}}
\newcommand{\iso}{\cong}

\newcommand{\im}{\operatorname{Im}}
\newcommand{\id}{\operatorname{id}}
\newcommand{\bra}{\langle}
\newcommand{\ket}{\rangle}

\newcommand{\Sq}{\operatorname{Sq}}
\newcommand{\inj}{\hookrightarrow}

\newcommand{\Gr}{\operatorname{Gr}}
\newcommand{\OGr}{\widetilde{\operatorname{Gr}}}
\newcommand{\Fl}{\operatorname{Fl}}

\newcommand{\al}{\alpha}
\newcommand{\be}{\beta}
\newcommand{\ga}{\gamma}
\newcommand{\Ga}{\Gamma}

\newcommand{\de}{\delta}

\newcommand{\ka}{\kappa}
\newcommand{\la}{\lambda}
\newcommand{\La}{\Lambda}
\newcommand{\si}{\sigma}

\newcommand{\se}{\subseteq}
\newcommand{\su}{\backslash}

\newcommand{\Z}{\mathbb{Z}}
\newcommand{\R}{\mathbb{R}}

\newcommand{\Q}{\mathbb{Q}}
\newcommand{\F}{\mathbb{F}}
\newtheorem{fact}{Fact}[section]
\newtheorem{lemma}[fact]{Lemma}
\newtheorem{theorem}[fact]{Theorem}
\newtheorem{defi}[fact]{Definition}
\newtheorem{exa}[fact]{Example}
\newtheorem{cla}[fact]{Claim}
\newtheorem*{sol}{\it Solution}
\newtheorem{proposition}[fact]{Proposition}
\newtheorem{corollary}[fact]{Corollary}
\newtheorem{conjecture}[fact]{Conjecture}

\newenvironment{definition}{\begin{defi} \rm}{\end{defi}}
\newenvironment{example}{\begin{exa} \rm}{\end{exa}}

\theoremstyle{remark}
\newtheorem{nremark}[fact]{Remark}
\newenvironment{remark}
{\par\pushQED{\qed}\nremark \small}
{\popQED\endnremark}

\renewcommand{\L}{\mathscr{L}}

\newcommand{\OFl}{\widetilde{{\rm Fl}}}

\newcommand{\crk}{\mathrm{crk}}
\newcommand{\het}{\operatorname{ht}}

\renewcommand{\H}{{\rm H}}

\begin{document}

\title[4-torsion cohomology classes in oriented Grassmannians]{4-torsion classes in the integral cohomology\\ of oriented Grassmannians}

\address{\'Akos K.\ Matszangosz, HUN-REN Alfr\'ed R\'enyi Institute of Mathematics, Re\'altanoda utca 13-15, 1053 Budapest, Hungary}
\email{matszangosz.akos@gmail.com}

\address{Matthias Wendt, Fachgruppe Mathematik und Informatik, Bergische Universit\"at Wuppertal, Gaussstrasse 20, 42119 Wuppertal, Germany}
\email{m.wendt.c@gmail.com}

\thanks{\'A. K. M. is supported by the Hungarian National Research, Development and Innovation Office, NKFIH K 138828 and NKFIH PD 145995.}


\begin{abstract}
  We investigate the existence of 4-torsion in the integral cohomology of oriented Grassmannians. We prove a general criterion for the appearance of 4-torsion classes based on (twisted) Steenrod squares and show that there are many cases where this criterion is satisfied for minimal-degree anomalous classes, assuming a conjecture on the characteristic rank. We also establish the upper bound in the characteristic rank conjecture for oriented Grassmannians $\OGr_k(n)$, and prove the equality in the cases $k=5, n=2^t-1,2^t$ and $k=6, n=2^t$. This provides infinitely many examples of oriented Grassmannians having 4-torsion in their integral cohomology. On the way, we clarify the relation between minimal-degree anomalous classes and results of Stong on the height of the first Stiefel-Whitney class $w_1$ in the mod 2 cohomology of real Grassmannians, for which we give an independent proof. We also establish some bounds on torsion exponents for the integral cohomology of oriented flag manifolds. Based on these findings and further computational evidence, we formulate a conjectural relationship between the torsion exponent in the integral cohomology of homogeneous spaces and their deficiency.
\end{abstract}

\author{\'Akos K.\ Matszangosz and Matthias Wendt}
\maketitle
%\Yboxdim{5pt}

\setcounter{tocdepth}{1}
\tableofcontents

\section{Introduction}

The purpose of the present paper is to deal with a rather specific aspect of the integral cohomology of the oriented Grassmannians, namely the torsion exponent. For the real Grassmannians $\Gr_k(n)$ of $k$-planes in $\R^n$, it is known by Ehresmann's theorem \cite{Ehresmann} that all torsion in the integral cohomology is 2-torsion, i.e., that $2{\rm Tor}\big( {\rm H}^*(\Gr_k(n),\mathbb{Z})\big)=0$. It is a natural question whether the same result holds for its double cover, the oriented Grassmannians $\OGr_k(n)$ of oriented $k$-planes in $\R^n$. Indeed, in all cases where the integral cohomology of oriented Grassmannians has been computed, all torsion is 2-torsion, cf.\ e.g.\ the work of Jovanovi\'c \cite{Jovanovic2022}. However, we show in the present paper that this is not the case in general. More precisely, we prove the following theorem, cf.~Theorem~\ref{cor:main}:

\begin{theorem}
  \label{thm:4torsion-intro}\,%Hack
  \begin{enumerate}
  \item For any $t\geq 4$, there is a nontrivial 4-torsion class in ${\rm H}^{2^t-1}(\OGr_5(2^t-1);\Z)$.
  \item For any $t\geq 4$, there is a nontrivial 4-torsion class in ${\rm H}^{2^t-1}(\OGr_6(2^t);\Z)$.
  \end{enumerate}
  In particular, there are infinitely many oriented Grassmannians $\OGr_k(n)$ having torsion of exact order~4 in their integral cohomology.
\end{theorem}

To show this result, we formulate a criterion for the existence of 4-torsion, cf.\ Proposition~\ref{prop:2tor_condition}, which is a version of the Bockstein cohomology adapted to the specific setting of oriented Grassmannians. This criterion allows to exhibit a link between 4-torsion and minimal degree anomalous classes\footnote{A class $x\in \H^*(\OGr_k(n);\F_2)$ is anomalous, if it cannot be written as a polynomial of Stiefel--Whitney classes of the tautological bundle, see Definition \ref{def:char_anomalous}.}: in Theorem~\ref{thm:main_general} we show that in a large number of cases, the existence of 4-torsion is implied by the \emph{characteristic rank conjecture} 
formulated in our previous paper \cite{OGr3}
on the characteristic rank\footnote{The characteristic rank is the largest degree $c$, such that $\H^{\leq c}(\OGr_k(n);\F_2)$ is generated by the Stiefel--Whitney classes of the tautological bundle, see Definition \ref{def:charrank}. In other words, there are no anomalous classes up to degree $c$.} of the mod 2 cohomology of the oriented Grassmannians, see Conjecture~\ref{new-amazing-conjecture}. 

This connection motivates further study of the characteristic rank conjecture. For this, we make precise the relation between the Koszul homology picture from \cite{OGr3} and the kernel of $w_1$. As a consequence, we can explicitly identify the generators of $\ker(w_1)$ corresponding to the ascended and descended generators discussed in \cite{OGr3}. Having this connection allows to deduce the upper bound part of the characteristic rank conjecture from a result of Stong \cite{Stong} on the height of $w_1$, cf.~Theorem~\ref{thm:stong-upper-bound}. We also provide an alternative Schubert-calculus proof of Stong's result, cf.\ Section~\ref{sec:koszul-vs-ker}. Via brute force inspection of Stiefel--Whitney monomials, we are then able to prove the characteristic rank conjecture for $k=5,n=2^t-1$, cf.~Theorem~\ref{thm:charrank}, and $k=6,n=2^t$, cf.~Corollary~\ref{cor:charrank2t}. This provides further evidence for the characteristic rank conjecture, as well as the required input to prove the main result Theorem~\ref{thm:4torsion-intro}. 

Even though we can prove some partial results toward the characteristic rank conjecture, a more conceptual understanding or plausible proof strategy is still missing. Similarly, we currently also lack a conceptual understanding of the origin of 4-torsion classes in the cohomology of oriented Grassmannians. Indeed, experimental evidence suggests that for $\OGr_k(n)$ with $k\geq 5$, occurrence of 4-torsion seems to be the generic situation, cf.~Remark~\ref{rem:experiments}. Based on our results and experiments, we formulate a conjecture that for $k\geq 5$, minimal anomalous classes should produce 4-torsion classes in many cases; for a precise formulation (which requires a more careful case distinction), see Conjecture~\ref{conj1}. Nevertheless, as the computer experiments discussed in Remark~\ref{rem:experiments} indicate, there appear to be many 4-torsion classes that are not directly connected to minimal anomalous classes. 

On the other hand, as far as the torsion exponent is concerned, Theorem~\ref{thm:4torsion-intro} is as bad as the situation can get. More generally, based on the recent proof that all torsion in the cohomology of real flag manifolds is 2-torsion, cf.~\cite{HudsonMatszangoszWendt}, we show that the torsion exponent of oriented partial flag varieties is bounded by the number of steps as follows, see Theorem~\ref{thm:torsion-bound-flags} and its Corollary: 

\begin{theorem}
  Let $\mathcal{D}=(d_1\stb d_m)$ be a sequence of positive integers, and denote by $\OFl_\mathcal{D}$ the oriented partial flag variety of flags $V_\bullet=(V_1\se V_2\se \cdots \se V_m=\R^{d_1+\cdots +d_m})$ where each subspace is oriented. Then we have  
  \[
  2^{m}{\rm{Tor}}\left({\rm H}^*(\OFl_{\mathcal{D}};\Z)\right)=0.
  \]
  In particular, all torsion in the integral singular cohomology of oriented Grassmannians is of order $2$ or $4$:
  \[4{\rm Tor}\left( {\rm H^*}(\OGr_k(n);\Z)\right)=0.\]
\end{theorem}

To show this bound on the torsion exponent requires a detour through sheaf cohomology theories arising in algebraic geometry, which already played a crucial part in the proof that all torsion in cohomology of real flag manifolds is 2-torsion \cite{HudsonMatszangoszWendt}. Essentially, the fundamental ideal filtration of sheaves of quadratic forms provides, via Jacobson's real cycle class map, additional structure on the singular cohomology of algebraic varieties that is not so easily accessible by purely topological means. 

As we see, we can get at most 4-torsion in the integral cohomology of oriented Grassmannians, Theorem~\ref{thm:4torsion-intro} shows that 4-torsion does indeed appear in examples, and as discussed earlier we also expect 4-torsion to appear abundantly for $k\geq 5$. There are, however, also some cases where only 2-torsion should be expected. At the moment, we don't fully understand the exact conditions ensuring that all torsion in the cohomology of a specific oriented Grassmannian $\OGr_k(n)$ is 2-torsion. Experimental evidence suggests that 4-torsion doesn't appear for $k\leq 4$, cf.\ Conjecture~\ref{conj1} and Remark~\ref{rem:experiments}.

More generally -- beyond the scope of oriented Grassmannians -- it would be desirable to have a description of the torsion exponents of a general homogeneous space $G/K$. To a homogeneous space and a finite coefficient field $\F_p$, one can associate a number, called its deficiency, due to Baum~\cite{Baum1968}. Based on the results of this paper, we formulate \emph{the deficiency conjecture}, which is a conjectural relationship between the $p$-torsion exponents and the $p$-deficiency of $G/K$, see Conjecture~\ref{conj:deficiency}. A special case of this conjecture states that whenever the deficiency with coefficients $\F_2$ is equal to 0, the $2$-primary torsion in the integer coefficient cohomology of $G/K$ consists of elements of order exactly 2. A proof of this deficiency conjecture together with the characteristic rank conjecture would then give an almost\footnote{The cases $k\geq 5$ odd and $2^{t-1}<n\leq \frac{k+1}{k}2^{t-1}$ are not covered by Conjecture \ref{conj1}  - even though we expect the appearance of 4-torsion classes in these cases as well, we do not have explicit candidates for such 4-torsion classes.} complete picture describing which $\OGr_k(n)$ have 4-torsion in their integral cohomology via Theorem~\ref{thm:main_general} and Proposition~\ref{prop:OGr_def}.

\subsubsection*{Structure of the paper:}

We start with some background on the mod 2 cohomology of oriented Grassmannians in Section~\ref{sec:basics-mod2}. We establish in Section~\ref{sec:exponent-bound} a general bound on torsion exponents for oriented flag manifolds based on recent work on algebraic cohomology theories related to quadratic forms. Then we discuss the relation between torsion exponents and deficiency of the cohomology algebra for homogeneous spaces in Section~\ref{sec:torsion-exponent}. We formulate the deficiency conjecture and discuss its relationship to the torsion exponents of oriented Grassmannians. In Section~\ref{sec:gysin} we establish general criteria for the existence of 4-torsion, based on integral Gysin sequences and Bockstein operations. We make precise the correspondence between the generators of Koszul homology and generators of the kernel of $w_1$ in Section~\ref{sec:koszul-vs-ker}, linking anomalous classes to the height of $w_1$, and we also provide a new Schubert calculus proof of a result of Stong on the height of $w_1$. These results are then used in Section~\ref{sec:charrank} where  we establish partial results toward the characteristic rank conjecture. The main computations for our 4-torsion examples are done in Section~\ref{sec:4torsion-examples}, checking the 4-torsion criterion for the ascended and descended generators in mod 2 cohomology, and thus linking the existence of 4-torsion to the characteristic rank conjecture.

\section{Mod 2 cohomology of oriented Grassmannians}
\label{sec:basics-mod2}

We first provide a brief recollection concerning the mod 2 cohomology of oriented Grassmannians and some facts about its ring structure. We also introduce the notation used later in the paper.

\subsection{Oriented Grassmannians as double covers: the Gysin sequence}

The oriented Grassmannian $\OGr_k(n)$ can be identified as the sphere bundle of the determinant bundle $\L=\det S_0$ of the tautological bundle $S_0\to\Gr_k(n)$. Indeed, given $\mathbb{R}^n$ equipped with the standard scalar product, a point of the total space of this sphere bundle consists of a subspace $W\in \Gr_k(n)$, together with a norm-preserving orientation $\det W\cong\mathbb{R}^\times$.
Therefore one of the natural tools to compute the cohomology of the oriented Grassmannians is the long exact Gysin sequence associated to $\L$:
\[\xymatrix{	\cdots \ar[r]&
	{\rm H}^{i-1}(\Gr_k(n);\F_2)\ar[r]^{w_1}& 
	{\rm H}^i(\Gr_k(n);\F_2) \ar[r]^{\pi^*}& 
	{\rm H}^i(\OGr_k(n);\F_2) \ar[r]^\de& 
	{\rm H}^{i}(\Gr_k(n);\F_2)\ar[r]&\cdots
}
\]
In particular, the cohomology of $\OGr_k(n)$ sits in the short exact sequence:
\begin{equation}\label{eq:w1ses}
	\xymatrix{	
		0\ar[r]&
		\coker w_1 \ar[r]^-{\pi^*}& 
		{\rm H}^*(\OGr_k(n);\F_2) \ar[r]^-\de& 
		\ker w_1\ar[r]&0
	}
\end{equation}
where $\de$ is a map of degree 0. 

\begin{definition}\label{def:char_anomalous}
	The graded ring $C^*=\coker w_1\subset {\rm H}^*(\OGr_k(n);\mathbb{F}_2)$ is called \emph{characteristic subring}. The classes in ${\rm H}^*(\OGr_k(n);\mathbb{F}_2)$ with non-trivial boundary are called \emph{anomalous classes}.
\end{definition}
Explicit presentations arise from characteristic class descriptions. For instance,
\begin{equation}\label{eq:HGr}
	{\rm H}^*(\Gr_k(n);\F_2)=\F_2[w_1\stb w_k]/(Q_{n-k+1}\stb Q_n)
\end{equation}
where $Q_i=w_i(\ominus S)$ for the universal rank $k$ bundle $S\to {\rm BO}(k)$. Since the characteristic subring is $C=\coker w_1$,
\begin{equation}\label{eq:C}
	C=\F_2[w_2\stb w_k]/(q_{n-k+1}\stb q_n)
\end{equation}
where $q_i=w_i(\ominus S)$ for the universal oriented rank $k$ bundle $S\to {\rm BSO}(k)$. Algebraically, $q_i=\rho (Q_i)$ for the reduction map $\rho\colon W_1\to W_2$, where for fixed $k$, we use the notation
\begin{equation}
	\label{eq:W1W2}
	W_1={\rm H}^*({\rm BO}(k))=\F_2[w_1\stb w_k],\qquad 
	W_2={\rm H}^*({\rm BSO}(k))=\F_2[w_2\stb w_k].
\end{equation}
We will also consider lifts of $q_i$ to $W_1$ via the natural ring inclusion $\iota\colon W_2\inj W_1$, and we will denote these lifts by
\begin{equation}\label{eq:tildeq}
	\tilde{q}_i:=\iota(q_i)
\end{equation}
Explicitly, we have
\begin{equation}\label{eq:qi}
	\tilde{q}_i=q_i=\sum_{2a_2+3a_3+\ldots+ka_k=i}\binom{|a|}{a}w^a	
\end{equation}
where the only difference between $q_i$ and $\tilde{q}_i$ is in where these classes live.

\subsection{Anomalous generators and the characteristic rank conjecture}
As noted above, the Gysin sequence yields a short exact sequence~\eqref{eq:w1ses} of $C$-modules. A first natural step toward understanding the mod 2 cohomology ring structure is the $C$-module structure on the kernel $K:=\ker w_1$. Partial information about this is the lowest nonzero degree of $K$, which is related to the characteristic rank of the tautological bundle over $\OGr_k(n)$. Below we recall the relevant definitions, as well a conjecture from \cite{OGr3} describing the characteristic rank of oriented Grassmannians. 

\begin{definition}
  \label{def:charrank}
  The \emph{characteristic rank} $\crk(E)$ of a real vector bundle $E$ of rank $n$ over a smooth manifold $M$ is the largest $k$, for which the classifying map $\ka^*\colon {\rm H}^{\leq k}({\rm BO}(n);\F_2)\to {\rm H}^{\leq k}(M;\F_2)$ is surjective. In other words, it is the largest $k$, such that all classes in ${\rm H}^{\leq k}(M)$ can be written as a polynomial in Stiefel--Whitney classes of $E$. A class $x\in {\rm H}^i(M;\F_2)$ is \emph{anomalous} (with respect to $E$), if $x\not\in \im \ka^*$.
\end{definition}

It is not hard to see that anomalous classes in $\OGr_k(n)$, cf.\ Definition~\ref{def:char_anomalous}, coincide with anomalous classes with respect to the tautological bundle $S\to \OGr_k(n)$. This definition implies that the lowest-degree anomalous class is in degree $\crk(S)+1$; the inclusion $C^r\subset {\rm H}^r(\OGr_k(n);\mathbb{F}_2)$ is an isomorphism for $r\leq \crk(S)$. In \cite{OGr3} we formulated the following conjecture.

\begin{conjecture}
  \label{new-amazing-conjecture}
  For $5\leq k\leq 2^{t-1}<n\leq 2^t$ and $t\geq 5$, the characteristic rank of the tautological bundle $S\to\OGr_k(n)$ is equal to
  \begin{equation}\label{eq:charrank}
    \crk(S)=\min (2^t-2,k(n-2^{t-1})+2^{t-1}-2).
  \end{equation}
\end{conjecture}

In this paper, we will show the right-hand side is indeed an upper bound, see Theorem~\ref{thm:stong-upper-bound}.

\subsection{Action of the Steenrod algebra and Bockstein cohomology}
Recall that for a (connected) topological space $X$, real line bundles $\L$ over $X$ can be identified with group homomorphisms $\pi_1(X)\to \mathbb{Z}/2\mathbb{Z}$, which in turn can be identified with rank one local systems on $X$. For a line bundle $\L$ on $X$, we will denote by ${\rm H}^*(X;\L)$ the cohomology of $X$ with local coefficients given by the associated rank one local system. We'll also call this the $\L$-twisted cohomology.\footnote{This is motivated by the analogous story on the algebraic side where certain sheaf cohomology theories, like Witt-sheaf cohomology or Chow--Witt groups, can be twisted by line bundles. For more on this story and real cycle class maps appearing later, see~\cite{HWXZ}.}

If $\L$ is a line bundle over $X$, then there is an associated Bockstein homomorphism to the $\L$-twisted cohomology
\[
\be_\L\colon\H^*(X;\F_2)\to \H^{*+1}(X;\L).
\]
The usual Bockstein homomorphism is $\be=\be_{\mathscr{O}}$, for the trivial line bundle (in which case $\H^*(X;\mathscr{O})\cong\H^*(X;\mathbb{Z})$). The mod 2 reduction of the Bockstein homomorphism $\be_\L$ is the twisted first Steenrod square
\[
\Sq^1_\L\colon\H^*(X;\F_2)\to \H^{*+1}(X;\F_2).
\]

Let us recall some elementary properties of $\Sq^1$ and $\Sq^1_\L$. Setting $w_1=w_1(\mathscr{L})$, we have by definition $\Sq^1_\L(x)=w_1x+\Sq^1(x)$, or 
\begin{equation}
	w_1=\Sq^1+\Sq^1_\L.
\end{equation}

\begin{proposition}\label{prop:Sq1_properties}
  The following commutation relations hold:
  \begin{itemize}
  \item $\Sq^1\circ w_1=w_1\circ \Sq^1_\L=\Sq^1\circ \Sq^1_\L$
  \item $\Sq^1_\L\circ \Sq^1=\Sq^1_\L \circ w_1=w_1\circ \Sq^1$
  \item $\Sq^1\circ \Sq^1=\Sq^1_\L\circ \Sq^1_\L=0$
  \end{itemize}
\end{proposition}

\begin{proof}
	Since $\Sq^1$ is a derivation,
	\[\Sq^1\circ w_1=w_1^2+w_1\circ \Sq^1=w_1\circ \Sq^1_\L,\] 
	which is also equal to $\Sq^1\circ \Sq^1_\L$ using the vanishing relations
	$\Sq^1\circ \Sq^1=0$ and
	\[\Sq^1_\L\circ\Sq^1_\L=(w_1+\Sq^1)\circ \Sq^1_\L=w_1\circ \Sq^1_\L+\Sq^1\circ \Sq^1_\L=0.\]
	The proof of the other equality is entirely analogous.
\end{proof}

For future reference, let us note that the short exact sequence \eqref{eq:w1ses} is in fact a sequence of modules for the Steenrod algebra. 

\begin{proposition}
  \label{prop:steenrod-closed}
  In the short exact sequence \eqref{eq:w1ses}, $C=\coker w_1$ and $K=\ker w_1$ are Steenrod-modules, and the maps are Steenrod-module homomorphisms.
\end{proposition}

\begin{proof}
  For the first half of the statement, it is enough to note that the ideal $(w_1)$ is a Steenrod submodule:
  \[\Sq^k(x\cdot w_1)=\Sq^{k}(x)\cdot w_1+\Sq^{k-1}(x)\cdot w_1^2.\]
  For the second statement, the first map is a cohomological pullback and therefore compatible with Steenrod operations. The second morphism (projection to $K$) factors as a composition
  \[
    {\rm H}^i(\OGr_k(n);\F_2)\xrightarrow{\partial}{\rm H}^{i+1}(\L,\L\setminus 0;\F_2)\xrightarrow{{\rm Th}}{\rm H}^i(\Gr_k(n);\F_2).
  \]
  Here, the first map is the boundary map in the long exact sequence for the pair $(\L,\L\setminus 0)$ of the total space of the determinant line bundle $\L$ and the complement $\L\setminus 0$ of the zero section. The second map is the Thom isomorphism for the line bundle $\L$. Stability of the Steenrod operations means that they commute with the boundary map $\partial$. In general, Steenrod operations don't commute with the Thom isomorphism. Rather ${\rm Th}\circ{\rm Sq}\circ{\rm Th}^{-1}(1)=w$ is the total Stiefel--Whitney class, and therefore ${\rm Th}\circ{\rm Sq}\circ{\rm Th}^{-1}=w\cdot{\rm Sq}$. The key point in our case is that the total Stiefel--Whitney class is $w=1+w_1$, and any element in the image of $\delta$ is in $K=\ker(w_1)\subseteq {\rm H}^i(\Gr_k(n);\F_2)$. This implies that the Thom isomorphism commutes with Steenrod squares for elements in the image of $\partial$, and consequently the projection ${\rm H}^*(\OGr_k(n);\F_2)\to K$ is compatible with Steenrod operations. 
\end{proof}

\subsection{Stabilization properties of Grassmannians}
We recall some stabilization properties of the inclusions of Grassmannians.

\begin{lemma}\label{lemma:push-pull}
  Let $i\colon \Gr_k(n)\to \Gr_k(n+j)$ be the natural map induced by the linear inclusion $\iota:\R^n\to \R^{n+j}$. Then 
  \[
  \ker i^* = (Q_{n-k+1}\stb Q_{n-k+j})\se {\rm H}^*(\Gr_k(n+j)),\qquad \ker i_!=(0)
  \]
  
 and if $w^a=w_1^{a_1}\ldots w_k^{a_k}$ is a Stiefel--Whitney monomial, then
  \[
  i_!w^a=w_k^j\cdot w^a.
  \]
\end{lemma}

\begin{proof}
  We first note that the tautological sub-bundle $S_{n+j}$ over $\Gr_k(n+j)$ pulls back via $i$ to the tautological bundle $S_n$ over $\Gr_k(n)$. Since $\H^*(\Gr_k(n))$ is generated by $w_l(S_{n})=i^*w_l(S_{n+j})$, for $l=1\stb k$, we find that $i^*$ is surjective. On the other hand, $I=(Q_{n-k+1}\stb Q_{n-k+j})\se \ker i^*$, since the pull-back of the quotient bundle splits off a rank $j$ trivial bundle, so its top $j$ Stiefel-Whitney classes $Q_{n-k+1}\stb Q_{n-k+j}$ are zero. Since by \eqref{eq:HGr}, 
  \[
  \underbrace{\F_2[w_1\stb w_k]/(Q_{n-k+1}\stb Q_{n})}_{\H^*(\Gr_k(n))}\iso \underbrace{\F_2[w_1\stb w_k]/(Q_{n+j-k+1}\stb Q_{n+j})}_{\H^*(\Gr_k(n+j))}\big/(Q_{n-k+1}\stb Q_{n-k+j})
  \]
 we can conclude that $\ker i^*\se I$ and therefore $\ker i^*=I$. 
	
  For the second part of the statement, choose flags $F_\bullet$ in $\R^n$ and $E_\bullet$ in $\R^{n+j}$ compatible with $\iota$, in the sense that $\iota(F_l)=E_l$ for $l\leq n$. Then by the definition of Schubert varieties (e.g.\ \cite[9.4]{Fulton}), we have $i(\si_\la(F_\bullet))=\si_{\la'}(E_\bullet)$, where $\la'=(\la_1+j,\la_2+j\stb \la_k+j)$, so 
  \[i_![\si_\la]=[\si_{\la'}].\]
  Since the cohomology has a basis of Schubert classes, this implies that the pushforward $i_!$ is injective. 
  Finally, the normal bundle of $i$ is $\Hom(S,\R^j)$, whose Euler class is $w_k^j$, which implies by the adjunction formula that $i_!w^a=i_!i^*w^a=w^a\cdot i_!1=w^a\cdot w_k^j$.
\end{proof}

\begin{lemma}
  \label{lemma:pushSq1}
  Let $i\colon\Gr_k(n-1)\to \Gr_k(n)$. Then 
  \[
  i_!\Sq^1=\Sq^1_\L i_!,\qquad 	i_!\Sq^1_\L=\Sq^1 i_!.
  \]
\end{lemma}

\begin{proof}
  By a theorem of Atiyah and Hirzebruch (a Grothendieck--Riemann--Roch-type statement for Steenrod operations), cf. \cite[Satz 3.2]{AtiyahHirzebruch}, we have 
  \[\Sq(i_! x)=i_!\bigl(\Sq(x)\cdot w(\nu_i)\bigr),\]
 where $w(\nu_i)$ denotes the total Stiefel-Whitney class of the normal bundle of the inclusion $i$. After taking the appropriate degree part, we get:
  \[\Sq^1i_!x=i_!(x\cdot w_1(\nu_i))+i_!\Sq^1x=i_!\Sq^1_\L x,\]
  where we use that $w_1(\nu_i)=w_1(S_{n-1})=w_1(\mathscr{L})$. Similarly,
  \[\Sq^1_\L i_!x=
  w_1\cdot i_!x+\Sq^1 i_!x=i_! (w_1\cdot x)+i_!\Sq^1_\L x=i_!\Sq^1 x.\qedhere
  \]
\end{proof}

\section{Bounding the torsion exponent for oriented flag manifolds}
\label{sec:exponent-bound}

Before discussing cases where 4-torsion appears in the integral cohomology of oriented Grassmannians, we want to establish a priori bounds on torsion. For this, we're actually going to back up a bit and establish more generally a bound on the torsion exponent for oriented flag manifolds $\OFl_{\mathcal{D}}$ and related coverings of partial flag manifolds.

To set things up, let $\mathcal{D}=(d_1,\dots,d_r)$ be a tuple of positive integers with $\sum d_i=N$. Then the oriented (real) flag manifold $\OFl_{\mathcal{D}}$ is the manifold of flags $V_\bullet=(V_1\se V_2\se \ldots \se V_r=\R^N)$ where each subspace $V_i$ is oriented and has dimension $\sum_{j=1}^i d_j$. It can be identified as homogeneous space ${\rm SO}(N)/({\rm SO}(d_1)\times\cdots\times {\rm SO}(d_r))$. 

For the bound on the torsion exponent, we need to realize the oriented flag manifolds as algebraic varieties. First, we can write $\OFl_{\mathcal{D}}$ as iterated degree 2 covering space of the ordinary flag manifold ${\rm Fl}_{\mathcal{D}}$ as follows. Start with $X_1={\rm Fl}_{\mathcal{D}}$, and fix a nontrivial real line bundle $\L_1$ on $X_1$. Associated to $\L_1$ is a degree 2 covering $p\colon X_2\to X_1$ such that $p^*(\L_1)$ is trivial. Take a nontrivial real line bundle $\L_2$ on $X_2$ and repeat. The process terminates since ${\rm Pic}(\Fl_{\mathcal{D}})/2\cong{\rm H}^1(\Fl_{\mathcal{D}};\F_2)\cong \mathbb{Z}/2\mathbb{Z}^{\oplus (r-1)}$. The result is a tower $\OFl_{\mathcal{D}}=X_r\to X_{r-1}\to\cdots\to X_2\to X_1=\Fl_{\mathcal{D}}$ of degree 2 coverings.\footnote{One possible choice for the line bundle $\L_j$ on $X_j$ is the pullback of the determinant bundle of the subbundle with fiber $V_j$. For this choice, the points of any $X_i$ correspond to flags where the first subspaces $V_1,\dots,V_{i-1}$ have been equipped with an orientation.}

To get an algebraic realization of the oriented flag manifold, we replace the degree 2 covering for a line bundle $\L_j$ on $X_j$ by the complement of the zero section of $\L_j$. Up to isomorphism, this doesn't change the cohomology since the degree 2 covering is a deformation retract of the complement of the zero section (each fiber $\mathbb{R}^\times$ is deformation retracted to $\{\pm 1\}$). The result is now a tower
\begin{equation}
  \label{eq:tower}
  \OFl_{\mathcal{D}}\simeq X_r'\to X_{r-1}'\to\cdots\to X_2'\to X_1=\Fl_{\mathcal{D}}
\end{equation}
of $\mathbb{R}^\times$-fiber bundles, where each $X_j'$ is the manifold of real points of a quasi-projective real variety. As a homogeneous space, the algebraic realization of the oriented flag manifold is the quotient ${\rm SL}_n/P$ with $P$ a parabolic subgroup of block-upper triangular matrices, whose Levi subgroup is the block-diagonal matrix group ${\rm SL}_{d_1}\times\cdots\times{\rm SL}_{d_r}$.

To establish the torsion bound, we now use Jacobson's real cycle class map which relates a certain sheaf cohomology ${\rm H}^*_{\rm Zar}(X;{\bf I}^q(\mathscr{L}))$ on a real algebraic variety $X$ with singular cohomology ${\rm H}^*_{\rm sing}(X(\mathbb{R});\mathbb{Z}(\mathscr{L}))$ of the space of real points. The sheaves ${\bf I}^q$ appearing here are the Zariski sheaves of powers of fundamental ideals in Witt rings of quadratic forms. In the following, we will freely use some of the basic facts on ${\bf I}^q$-cohomology. Most importantly, we use results on the ${\bf I}^q$-cohomology for suitably cellular varieties from \cite{HWXZ} and \cite{hennig}, as well as the computations for flag varieties in \cite{HudsonMatszangoszWendt}. For further information on ${\bf I}^q$-cohomology, cf.\ \cite{HWXZ} and \cite{HudsonMatszangoszWendt}. 

\begin{theorem}
  \label{thm:torsion-bound-flags}
  Let $\mathcal{D}=(d_1\stb d_r)$ be a sequence of positive integers with $\sum d_j=N$, and denote by $\OFl_{\mathcal{D}}$ the oriented partial flag variety of flags $V_\bullet=(V_1\se V_2\se \ldots \se V_r=\R^N)$ where each subspace $V_i$ is oriented and has dimension $\sum_{j=1}^id_j$. Then we have  
  \[
  2^{r}{\rm{Tor}}\left({\rm H}^*(\OFl_{\mathcal{D}};\Z)\right)=0.
  \]
\end{theorem}

\begin{proof}
  We will use the algebraic realization of the oriented flag varieties and intermediate $X_j'$ in the tower~\eqref{eq:tower} above, and prove the bound on torsion exponent for cohomology of $X'_i$ with local coefficients by an induction on $i$. We claim that for any $i$ and any line bundle $\L$ on $X'_i$ 
  \[
  2^i{\rm Tor}\left({\rm H}^*(X'_i,\mathbb{Z}(\L))\right)=0.
  \]
  This will, in particular, establish the claim of the theorem, but it will indeed show that torsion bounds are even satisfied for ``partially oriented partial flag manifolds'', with local coefficients in rank one local systems.
  
  The base case is $X'_1=\Fl_{\mathcal{D}}$. In this case, our claim is that all torsion in ${\rm H}^*(X'_1,\mathbb{Z}(\L))$ is 2-torsion, which follows from~\cite[Theorem~1.3]{HudsonMatszangoszWendt}. As we will modify the argument here, we briefly outline the proof in loc.cit. It is based on the fact that partial flag varieties have cellular structures with affine space cells. Then \cite[Theorem~5.7]{HWXZ} provides an isomorphism
  \[
  {\rm H}^j(\Fl_{\mathcal{D}};{\bf I}^j(\L))\to {\rm H}^j_{\rm sing}(\Fl_{\mathcal{D}};\mathbb{Z}(\L))
  \]
  for each $j$. On the algebraic side, we then have long exact sequences
  \[
  \cdots\to {\rm H}^q(X,{\bf I}^{j+1}(\L))\to {\rm H}^q(X,{\bf I}^j(\L))\to {\rm H}^q(X,{\bf I}^j/{\bf I}^{j+1})\to {\rm H}^{q+1}(X,{\bf I}^{j+1}(\L))\to\cdots
  \]
  associated to the short exact sequences of sheaves $0\to {\bf I}^{j+1}(\L)\to {\bf I}^j(\L)\to {\bf I}^j/{\bf I}^{j+1}\to 0$. The quotient sheaves can be identified more precisely, using the theorem of Orlov--Vishik--Voevodsky (solution of the Milnor conjecture on quadratic forms):
  \[
    {\bf I}^j/{\bf I}^{j+1}\cong{\bf K}^{\rm M}_{j+1}/2
  \]
  These sheaves are 2-torsion, and so are their cohomology groups. The main point of \cite{HudsonMatszangoszWendt} is then to show ${\rm H}^q(\Fl_{\mathcal{D}};{\bf I}^j(\L))$ for $q>j$ are torsion-free, which implies in particular that the torsion in ${\rm H}^j(\Fl_{\mathcal{D}};{\bf I}^j(\L))$ is exactly the 2-torsion coming from the image of the Bockstein map. 

  Now we want to establish a similar torsion bound for $X'_i$. We first note that the $X'_i$ has a stratification by subspaces of the form $\mathbb{A}^{d}\times\mathbb{G}_{\rm m}^{i-1}$, with varying $d$. For the flag variety $X'_1=\Fl_{\mathcal{D}}$ this is the classical stratification by Schubert cells. For the $X'_j$ it follows by induction: If it is true for $X'_j$, then the line bundle $\L_j$ will be trivial over the cells of $X'_j$, so the preimage of a cell $C$ in $X'_j$ under the map $X'_{j+1}\to X'_j$ will simply be $C\times\mathbb{G}_{\rm m}$. This provides the cell structure for $X'_{j+1}$. An extension of \cite[Theorem~5.7]{HWXZ} to such cellular structures has been established in the upcoming PhD thesis of Jan Hennig~\cite{hennig}, implying that for $X'_j$ we have isomorphisms
  \[
  {\rm H}^q(X'_j;{\bf I}^{q+j-1}(\L))\xrightarrow{\cong} {\rm H}^q_{\rm sing}(X'_j;\mathbb{Z}(\L)). 
  \]
  Using the long exact sequences
  \[
  \cdots\to {\rm H}^q(X,{\bf I}^{j+1}(\L))\to {\rm H}^q(X,{\bf I}^j(\L))\to {\rm H}^q(X,{\bf I}^j/{\bf I}^{j+1})\to {\rm H}^{q+1}(X,{\bf I}^{j+1}(\L))\to\cdots
  \]
  together with the fact that the groups ${\rm H}^q(X,{\bf I}^j/{\bf I}^{j+1})$ are 2-torsion shows that we get the required torsion bound if we can show that ${\rm H}^q(X'_j,{\bf I}^{q-1}(\L))\cong{\rm H}^q(X'_j,{\bf W}(\L))$ is torsion-free. This again follows inductively. For $X'_1$, this is the main result of \cite{HudsonMatszangoszWendt}. Assuming it is true for $X'_j$, we can use the Gysin sequence
  \[
  \cdots\to {\rm H}^q(X'_j,{\bf W}(\L))\to {\rm H}^q(X'_{j+1},{\bf W}(\L))\to {\rm H}^q(X'_j;{\bf W}(\L\otimes\L_j))\xrightarrow{e(\L_j)} {\rm H}^{q+1}(X'_j,{\bf W}(\L))\to \cdots
  \]
  Here $e(\L_j)$ is the Euler class of the line bundle $\L_j$ on $X'_j$, and this is 0 in Witt-sheaf cohomology. By the torsion-freeness assumption, the resulting short exact sequences
  \[
  0\to {\rm H}^q(X'_j,{\bf W}(\L))\to {\rm H}^q(X'_{j+1},{\bf W}(\L))\to {\rm H}^q(X'_j;{\bf W}(\L\otimes\L_j))\to 0
  \]
  split, showing torsion-freeness for the cohomology of $X'_{j+1}$.
\end{proof}

\begin{corollary}
  \label{cor:ogr-torsion-bound}
    In particular, all torsion in the integral singular cohomology of oriented Grassmannians is of order $2$ or $4$:
  \[4{\rm Tor}\left( {\rm H^*}(\OGr_k(n);\Z)\right)=0.\]
\end{corollary}

\begin{remark}
  For the oriented Grassmannians, the bound on torsion exponent in Corollary~\ref{cor:ogr-torsion-bound} can also be established by purely topological means, using the Gysin sequence for the double cover $\OGr_k(n)\to \Gr_k(n)$:
  \[
  \xymatrix{
    \cdots\ar[r]&
             {\rm H}^{i-1}(\Gr_k(n);\Z)\ar[r]^{e(\mathscr{L})}& 
             {\rm H}^i(\Gr_k(n);\L) \ar[r]^{\pi_\L^*}& 
             {\rm H}^i(\OGr_k(n);\mathbb{Z}) \ar[r]^{\de_{\L}}& 
             {\rm H}^{i}(\Gr_k(n);\Z)\ar[r]&\cdots 
  }.
  \]
  However, this argument doesn't generalize for other partial flag varieties. Using the Gysin sequence 
  \[
  \cdots
  \xrightarrow{e(\L_j)} {\rm H}^q(X_j;\mathbb{Z}(\L)) \xrightarrow{\pi_{\L_j}^*} {\rm H}^q(X_{j+1};\Z(\L)) \xrightarrow{\de_{\L_j}} {\rm H}^{q}(X_j;\Z(\L\otimes\L_i))\to\cdots 
  \]
  for the double cover $p\colon X_{j+1}\to X_j$, if the torsion in the cohomology of $X_j$ divides $2^m$, then the above sequence only shows that the torsion for $X_{j+1}$ divides $4^m$, because there could be nontrivial extensions
  \[
  0\to \mathbb{Z}/2^m\mathbb{Z}\to \mathbb{Z}/4^m\mathbb{Z}\to \mathbb{Z}/2^m\mathbb{Z}\to 0
  \]
  In particular, we wouldn't get the stronger bound we get from the algebraic argument of Theorem~\ref{thm:torsion-bound-flags}. It would be interesting to know if there is a ``more topological'' proof of Theorem~\ref{thm:torsion-bound-flags}, that doesn't require real algebraic geometry.
\end{remark}

\begin{remark}
  We offer a brief remark on algebraic realizability of singular cohomology classes and weight filtrations on cohomology. Jacobson's theorem that the real cycle class map
  \[
  {\rm H}^i_{\rm Zar}(X;{\bf I}^q)\to {\rm H}^i_{\rm sing}(X(\mathbb{R});\mathbb{Z})
  \]
  for a real algebraic variety $X$ is an isomorphism for $q>\dim X$ means that all classes in singular cohomology of a real variety are realizable by algebraic cycles \emph{if we allow more general quadratic form coefficients}. Stronger statements hold for cellular varieties, cf.\ \cite[Section 5]{HWXZ}, in which case the real cycle class map above is an isomorphism for $q\geq i$. Note that the cellular situation of \cite{HWXZ} is a situation in which also the complex cycle class map ${\rm CH}^i(X)\to {\rm H}^{2i}(X(\mathbb{C});\mathbb{Z})$ is an isomorphism.

  In general, the images of the real cycle class maps ${\rm H}^i_{\rm Zar}(X;{\bf I}^q)\to {\rm H}^i_{\rm sing}(X(\mathbb{R});\mathbb{Z})$ for $q$ between $i-1$ and $\dim X$ provide a filtration of singular cohomology reminiscent of filtrations in Hodge theory. These filtrations could potentially be used to better understand the 2-power torsion in cohomology of real algebraic varieties. One example is Theorem~\ref{thm:torsion-bound-flags} above, where the filtration is used to establish a bound on the exponent of the torsion.

  There are analogous filtrations on the mod 2 cohomology of algebraic varieties, induced by cycle class maps
  \[
  {\rm H}^*_{\rm Zar}(X;{\bf I}^q/{\bf I}^{q+1})\to {\rm H}^*_{\rm sing}(X(\mathbb{R});\F_2)
  \]
  for a real variety $X$. The coefficient sheaves ${\bf I}^q/{\bf I}^{q+1}\cong{\bf K}^{\rm M}_q/2$ are Milnor K-theory sheaves (by Orlov--Vishik--Voevodsky). These cycle class maps generalize the classical Borel--Haefliger map \cite{BorelHaefliger}
  \[
    {\rm H}^q_{\rm Zar}(X;{\bf K}^{\rm M}_q/2)\cong{\rm Ch}^q(X)\to {\rm H}^q_{\rm sing}(X(\mathbb{R});\F_2).
  \]
  For $q>\dim X$ the cycle class maps are isomorphisms by a theorem of Colliot-Th\'el\`ene and Scheiderer. From \cite{HWXZ} and \cite{hennig}, we find that the cycle class maps are isomorphisms on ${\rm H}^q({\bf I}^{q+r})$ for schemes with a cellular structure whose cells are of the form $\mathbb{A}^d\times\mathbb{G}_{\rm m}^{\times r}$. As before, the images of the cycle class maps for varying $q$ provide a filtration of mod 2 singular cohomology. For cellular varieties, we get similar bounds on the length of this filtration as in the integral case of Theorem~\ref{thm:torsion-bound-flags}.   It would be interesting to understand the relation between this filtration and weight filtrations in the work of McCrory and Parusi\'nski. Possibly this could be one approach to prove Theorem~\ref{thm:torsion-bound-flags} in a more topological way.

  For the specific case of oriented Grassmannians, the filtration has only one nontrivial step. The real cycle class map ${\rm H}^q_{\rm Zar}(\OGr_k(n);{\bf K}^{\rm M}_{q+1}/2)\to {\rm H}^q_{\rm sing}(\OGr_k(n);\F_2)$ is an isomorphism. The nontrivial subspace in the filtration is the image of the Borel--Haefliger cycle class map
  \[
    {\rm H}^q_{\rm Zar}(\OGr_k(n);{\bf K}^{\rm M}_q/2)\cong{\rm Ch}^q(\OGr_k(n))\to {\rm H}^q_{\rm sing}(\OGr_k(n);\F_2).
  \]
  In this case, it turns out that the image of the Borel--Haefliger map is exactly the characteristic subring, as ${\rm Ch}^*(\OGr_k(n))$ is the cokernel of multiplication by $w_1$ on ${\rm Ch}^*(\Gr_k(n))$ by the localization sequence for Chow groups. In particular, only the classes in the characteristic subring are fundamental classes of closed subvarieties, the anomalous classes can only be realized by ``higher weight'' cycles in ${\rm H}^q_{\rm Zar}(\OGr_k(n);{\bf K}^{\rm M}_{q+1}/2)$.
\end{remark}

\section{Torsion exponents for homogeneous spaces}
\label{sec:torsion-exponent}

In this section we summarize some general methods to compute the cohomology of homogeneous spaces $G/K$. We are interested in potential relations between cohomology ring structure and existence of torsion (or bounds on torsion exponents). We discuss in particular Baum's definition \cite{Baum1968} of the deficiency of a pair $(G,K)$, and relate the deficiency to bounds for the torsion exponent in the case of oriented Grassmannians. Based on the oriented Grassmannian case, we suggest a general picture relating deficiency and torsion exponent of general homogeneous spaces $G/K$, see Conjecture~\ref{conj:deficiency}.

\subsection{Torsion coefficients in homogeneous spaces}

Let $K\leq G$ be an inclusion of real Lie groups, and consider the homogeneous space $G/K$. The Betti numbers of such a homogeneous space $G/K$ are completely understood with field coefficients by the work of Cartan \cite{Cartan1951}, Borel \cite{Borel1953} and Baum \cite{Baum1968}. In contrast, the additive structure of integral cohomology ${\rm H}^*(G/K;\Z)$ is much less understood. Excluding $p$-primary torsion for different primes $p$ is possible by considering the cohomology of $G$ and $K$, however the actual torsion exponents are less readily available.

In principle, the torsion exponents can be computed from the Bockstein spectral sequence. If the $p$-Bockstein spectral sequence degenerates at the $E_r$-page, then all $p$-primary torsion is of order dividing $p^{r-1}$. If all $p$-primary torsion is $p$-torsion, the additive structure of integral cohomology can be pieced together from the $\F_p$ and $\Q$-coefficient Betti numbers. Although showing the degeneration of the Bockstein spectral sequence can be possible on a case-by-case basis, in general, only a few structural results are available.

Possibly the first general result bounding torsion exponents is due to Ehresmann \cite{Ehresmann}, who showed that all torsion in the cohomology of real Grassmannians is of order two. In modern terms, he showed the degeneration of the Bockstein spectral sequence at the $E_2$-page using the boundary coefficient description of Schubert classes (even though Steenrod squares and Bockstein operations had not been discovered at that point). Ehresmann's result generalizes to real partial flag manifolds of type A -- \cite{Matszangosz}, \cite{Yang}, \cite{HudsonMatszangoszWendt}, however the last two results make use of the algebraic structure of the flag manifolds and the degeneration of the Bockstein spectral sequence only follows indirectly.

\subsection{Deficiency}
We give a brief overview of the relevant notions from Baum's theory, and for further details we refer to Baum's original paper \cite{Baum1968}.

Let $k$ be a field. In this section we will consider finitely generated, graded-commutative $k$-algebras $A$, with $A^{<0}=0$ and $A^0=k$ (such that the $k$-algebra structure agrees with the $A^0$-algebra structure).  Let 
\[
Q(A)=A^{>0}/(A^{>0}\cdot A^{>0})
\]
denote the graded vector space of \emph{indecomposable elements}. A \emph{presentation of $A$} is an exact sequence\footnote{Exact here is meant in the sense that $\ker f_n=(\im f_{n+1})^{>0}$, generated by the positive degree part of $\im f_{n+1}$ -- in Baum's terminology, such sequences are called co-exact.}
\[
\xymatrix{
\La\ar[r]&\Ga\ar[r]^{f}& A\ar[r]& k,
}
\]
such that $\La$ and $\Ga$ are graded polynomial algebras, and the induced map 
$Q(\La)\to k\otimes_\La \ker f$ is an isomorphism of graded vector spaces (this condition ensures that there are no redundant relations). Our main cases of interest are presentations of the form
\begin{equation}\label{eq:coexact}
  \xymatrix{
    \H_G^*\ar[r]^\rho&\H_K^*\ar[r]& \H_K^*/(\im \rho)^{>0}\ar[r]& k,
  }
\end{equation}
for $G={\rm SO}(n)$ and $K={\rm SO}(k)\times{\rm SO}(n-k)$. Note that in this case $\H_K^*/(\im \rho)^{>0}$ is not the cohomology ring of $\OGr_k(n)=G/K$, but rather the characteristic subring $C$, cf. Definition~\ref{def:char_anomalous}.

\begin{example}\label{ex:OGr24}
  Let $G={\rm SO}(4)$, $K={\rm SO}(2)\times {\rm SO}(2)$ as a diagonal subgroup. The coefficient field for cohomology is $k=\F_2$. Recall the notation introduced in \eqref{eq:C}. Then the restriction map
  \[\rho\colon\F_2[\tilde w_2,\tilde w_3,\tilde w_4]\to \F_2[w_2,q_2]\]
  determined by
  \[\rho(1+\tilde w_2+\tilde w_3+\tilde w_4)=(1+w_2)(1+q_2)\]
  maps $\tilde w_2\mapsto w_2+q_2$, $\tilde w_3\mapsto 0$ and $\tilde w_4\mapsto w_2q_2$. In this case, the co-exact sequence \eqref{eq:coexact} is not a presentation: the map
  $Q(\H_G^*)\to k\otimes_{\H_G^*} (\im \rho)^{>0}$ is
  \[\bar{\rho}\colon \F_2\bra \tilde w_2,\tilde w_3,\tilde w_4\ket \to \F_2\bra w_2+q_2,w_2q_2\ket \]
  which is clearly not an isomorphism. Instead, restricting $\rho$ to the subring  $\La$ generated by $\tilde w_2$ and $\tilde w_4$ gives a presentation.
\end{example}

To an algebra $A$ as above, one can assign an integer called its deficiency as follows.

\begin{definition}
  Let 
  \[
  \xymatrix{
    \La\ar[r]&\Ga\ar[r]^{f}& A\ar[r]& k,
  }
  \]
  be a presentation of $A$, as defined above. Then the \emph{$i$-th deficiency of $A$} is the integer ${\rm def}_i(A):=\dim_k Q(\La)^i-\dim_k Q(\Ga)^i$. The \emph{deficiency of $A$} is ${\rm def}(A):=\sum_i {\rm def}_i(A)$.
\end{definition}
See \cite[Theorem 4.6]{Baum1968} for the proof that this definition is independent of the chosen presentation. 

\begin{remark}
  Baum's paper \cite{Baum1968} makes several even-degree assumptions at the outset (already the definitions are only given for such rings). However this condition as well as many other conditions have been relaxed, see \cite{HusemollerMooreStasheff1974}, \cite{Munkholm1974} \cite{Wolf1977}, \cite{Franz2021}, \cite{Carlson2023}. In our case of $\mathbb{F}_2$-coefficients, all rings are commutative which allows to remove even-degree assumptions, most statements in Baum's paper reduce to statements about regular sequences in commutative rings.
\end{remark}

Now let $K\leq G$ be compact Lie groups with $G$ connected. Let $k$ be a fixed field - note that most subsequent quantities depend on the choice of $k$. Assume that the cohomologies of $\H_G^*$ and $\H_K^*$ are free polynomial algebras - this is satisfied in a large number of cases, e.g.\ always if $k$ has characteristic 0. When $k=\F_p$, by Quillen's theorem \cite[Corollary 7.8]{Quillen1971}, the number of generators of such a polynomial algebra $\H_G^*$ is given by the maximal rank of an elementary abelian $p$-group inside $G$; we will call this number the $p$-\emph{rank} of $G$. In our applications $p=2$.
\begin{definition}
  The \emph{deficiency $\de(G,K)$ of the pair $(G,K)$} is 
  \[
  \de(G,K)={\rm def}\bigl(\H_K^*/(\im \rho)^{>0}\bigr)
  \]
  where $\rho\colon \H_G^*\to \H_K^*$ is the restriction map, and $(\im  \rho)^{>0}$ is the ideal in $\H_K^*$ generated by the positive degree elements of $\im \rho$. When $k$ is not fixed, we denote $\de_p$ the deficiency over $\F_p$. We sometimes write $\de(G/K)$ instead of $\de(G,K)$.
\end{definition}

Baum proves the following bounds for the deficiency in \cite{Baum1968}: 
\begin{equation}\label{eq:def_bounds}
	0\leq \de(G,K)\leq \rk_k(G)-\rk_k(K).
\end{equation}

\begin{remark}
  The statement is Lemma~7.1 of \cite{Baum1968}. Note that conventions in Baum's paper require cohomology to be concentrated in even degrees, which is not satisfied in the cases of special orthogonal groups we are interested in. Nevertheless, the inequalities still hold. The first inequality $0\leq\de(G,K)$ follows by reference to 6.2 and 4.10 (and subsequently 3.7) in \cite{Baum1968}, and both 6.2 and 3.7 are true without even-degree hypotheses. In the second inequality, our statement \eqref{eq:def_bounds} is slightly different from Baum's in that we only consider the $k$-rank, so there is no need to compare to the Lie group rank, and the second inequality more directly follows from the exactness of
  \[
  \xymatrix{
    \H_G^*\ar[r]^\rho&\H_K^*\ar[r]& \H_K^*/(\im \rho)^{>0}\ar[r]& k.\qedhere
  }
  \]
\end{remark}

\begin{remark}
  In other words, if $\H_G^*$ and $\H_K^*$ are polynomial algebras, and $\ka=\rk_k (K)$, then the deficiency $\de(G,K)$ is $n-\ka$, where $n$ is the cardinality of a subset of non-redundant generators\footnote{The generators $a_1\stb a_r$ of an ideal $I$ of $A$ are a \emph{non-redundant set of generators}, if no proper subset of them generates $I$.} $\rho(x_i)$ of the ideal $(\im\rho)^{>0}$, where $\rho\colon \H_G^*\to \H_K^*$ is the restriction map and $x_i$ are polynomial generators of $\H_G^*$, see \cite[Lemma~4.5]{Baum1968}.
\end{remark}

\begin{example}
  In Example \ref{ex:OGr24}, the deficiency is $\rk_{\F_2}\La-\rk_{\F_2} \H_K^*=2-2=0$.
\end{example}

\begin{example}
  Let $G={\rm SO}(10)$ and $K={\rm SO}(5)\times {\rm SO}(5)$. The inclusion $i\colon K\to G$ induces
  \[
  \begin{split}
    &i^*(1+\tilde w_2+\tilde w_3+\tilde w_4+\tilde w_5 + \tilde w_6+\tilde w_7 + \tilde w_8+ \tilde w_9 +\tilde w_{10})=\\ =&(1+w_2+w_3+w_4+w_5)(1+q_2+q_3+q_4+q_5)=\\
    =&1+(w_2+q_2)+(w_3+q_3)+(q_4+w_2q_2+w_4)+(w_5+q_5+w_2q_3+q_2w_3)+(w_2q_4+w_3q_3+q_2w_4)\\
    &+(w_2q_5+q_2w_5+w_3q_4+q_3w_4)+(w_3q_5+w_4q_4+w_5q_3) +(w_4q_5+q_4w_5)+w_5q_5
  \end{split}
  \]
  The relations are as follows:
  \[
  q_2=w_2, \quad q_3=w_3, \quad q_4=w_2^2+w_4, \quad q_5=w_5,
  \]
  \[w_2^3=w_3^2,\quad w_3w_2^2=0,\quad w_4^2=w_4w_2^2,\quad w_5w_2^2=0,\quad w_5^2=0.\]
  One can show that all of the restrictions $i^*\tilde w_j$ are required to generate the ideal $(\im i^*)^{>0}$, so that \eqref{eq:coexact} is indeed a presentation. Therefore the deficiency is 
  \[\rk_{\F_2} \H_G^*-\rk_{\F_2} \H_K^*=9-8=1.\]
  There are no relations in degrees $\leq 10$, the first relation between $i^*\tilde w_k$ lives in degree 12 as follows:
  \[ w_3\tilde w_9+w_5\tilde w_7+(w_2w_5+w_3w_4)\tilde w_5+(w_2w_3w_4+w_2^2w_5+w_4w_5)\tilde w_3+(w_2w_3w_5+w_5^2+w_3^2w_4)\tilde w_2=0.\]
  (This can be obtained by noting that $w_5(w_3w_2^2)=w_3(w_5w_2^2)$ in the relations above, and rewriting them in terms of the restriction $\tilde{w}_i$'s.)
\end{example}

\subsection{Deficiency of oriented Grassmannians}

Before examining actual torsion phenomena, let us consider the deficiency of oriented and unoriented partial flag manifolds. We take $k=\F_2$ whenever dealing with $\op O(n)$ and $\op{SO}(n)$.

\begin{proposition}\label{prop:Gr_def_bounds}
  Let $k=\F_2$ and $\sum d_i=N$.
  \begin{itemize}
  \item[i)] For  $G=\op{O}(N)$, $K=\op O(d_1)\times \ldots \times \op O(d_r)$ we have $\de(G,K)=0$.
  \item[ii)] For $G=\op{SO}(N)$, $K=\op{SO}(d_1)\times \ldots \times \op{SO}(d_r)$ we have $0\leq \de(G,K)\leq r-1$.
  \end{itemize}
\end{proposition}

\begin{proof}
  Using $\rk (\op{O}(m))= m$, $\rk (\op{SO}(m))=m-1$, we get $\rk(G)=N$, and in i) $\rk(K)=N$, and in ii) $\rk(K)=N-r$, and we can apply \eqref{eq:def_bounds}.
\end{proof}

Recall the notation $q_i$ from \eqref{eq:HGr}, \eqref{eq:C}.
Understanding the Koszul homology of a presentation gives information on the deficiency:
\begin{proposition}\label{prop:OGrdef}
If there is a $W_2$-relation between $q_{n-k+1}\stb q_n$  of the form 
\begin{equation}\label{eq:relation}
	q_j=\sum \ga_i q_i ,
\end{equation}
for $n-k<j\leq n$ and some $\ga_i\in W_2$, then $\de(\OGr_k(n))=0$.
If there are no such $W_2$-relations, then $\de(\OGr_k(n))=1$. 
\end{proposition}

\begin{proof}
  We have an exact sequence
\[
\xymatrix{
\F_2[q_{n-k+1}\stb q_n]\ar[r]&\F_2[w_2\stb w_k]\ar[r]&C=\H_G^*/(\im\rho)^{>0}\ar[r]&k
}
\]
If there is a relation of the form \eqref{eq:relation}, then $q_j$ is a redundant generator expressible in terms of the other $q_i$. In this case, the exact sequence fails to be a presentation, but we obtain a presentation upon removing $q_j$. The number of remaining $q_i$ equals the number of $w_i$, hence the deficiency is $0$. On the other hand, if there is no relation of the form \eqref{eq:relation}, none of the $q_i$ is redundant, and the exact sequence is already a presentation. Consequently, the deficiency is 1 in this case.
\end{proof}

\begin{corollary}
  If $\crk(S\to \OGr_k(n))>t$, then there are no $W_2$-relations between the $q_j$'s up to degree $t+2$.
\end{corollary}

\begin{proof}
  A nontrivial $W_2$-relation between the $q_j$'s implies a nontrivial class in the Koszul homology $H_1(Q,W_2)$, for this statement and relevant notation, see~\cite{OGr3}, Section 5.2. A nonzero class in Koszul homology $H_1(Q,W_2)$ in degree $d$ implies that $\crk(S\to\OGr_k(n))\leq d-2$, see the beginning of Section 5.3 of loc.cit. Conversely, if $\crk(S)>t$, then nonzero Koszul homology classes have to have degree $>t+2$.
\end{proof}

Let us recall from \cite{OGr3}, that there is a generalization of the results of Fukaya and Korba\v s\cite{Fukaya}, \cite{Korbas2015}, which states that for arbitrary $k$, the following relation holds:
\begin{equation}\label{eq:amazing_relation}
  \sum_{i\text{ even}}w_iq_{2^t-i}=\sum_{1<i\text{ odd}}w_iq_{2^t-i}=0.
\end{equation}
Applying Proposition \ref{prop:OGrdef} to this relation, we obtain the deficiency of oriented Grassmannians in a number of cases:
\begin{proposition}
  \label{prop:OGr_def}
  \[\de(\OGr_2(n))=0,\]
  \[\de(\OGr_3(n))=\de(\OGr_4(n))=\begin{cases}
  0,\qquad &n=2^{t-3},2^{t-2},2^{t-1},2^t,\\
  1,\qquad &\text{else}.\\
  \end{cases}\]
  and for $k$ odd,
  \[
  \de(\OGr_k(2^t))=
  0,\qquad \text{ if } k \text{ is odd}.
  \]
\end{proposition}

\begin{proof}
  These follow by applying Proposition \ref{prop:OGrdef} to results from \cite{Fukaya}, \cite{Korbas2015},  \cite{BasuChakraborty2020}, \cite{OGr3}.

  First, we cover the cases $\de=0$.	For $k=2$, $q_{\text{odd}}=0$, which is a relation of the form \eqref{eq:relation} for all $n$.
	
	For $k=3$ and $k=4$, we have $q_{2^{t-3}}=0$  by \cite{Fukaya}, \cite{Korbas2015}. This is a relation of the form \eqref{eq:relation} for $k=3$ and $n=2^{t-3},2^{t-2},2^{t-1}$, and for $k=4$ and $n=2^{t-3}, 2^{t-2},2^{t-1}, 2^t$.
	
	For $k$ odd and $n=2^t$, \eqref{eq:amazing_relation} is a relation of the form \eqref{eq:relation}. 

        This covers all the cases when $\de=0$.	Now let us turn to the cases when $\de=1$. In a number of cases, there are no relations in degrees $\leq n$ at all.
	
	For $\OGr_3(n)$ with other $n$, there are no relations in degrees $\leq n$ by \cite{BasuChakraborty2020} or \cite[Theorem 1.1]{OGr3}, showing that the degrees of anomalous generators are $\geq n$ for $n\neq 2^{t-3},2^{t-2},2^{t-1},2^t$.
	
	For $\OGr_4(n)$ with other $n$, there are no relations in degrees $\leq n$ since the characteristic rank of $\OGr_4(n)$ for other $n$'s is always $>n-5$, cf.\ \cite[Theorem 6.6]{PrvulovicRadovanovic2019}.
\end{proof}

\begin{conjecture}
  \label{conj:ogr-def}
  In all the remaining cases not covered by Proposition \ref{prop:OGr_def}, the deficiency is equal to 1. Explicitly, we conjecture that $\de(\OGr_k(n))=1$ holds for $5\leq k\leq 2^t-5$ unless we are in the case where $k$ is odd and $n=2^t$.
\end{conjecture}

Using Proposition \ref{prop:OGrdef}, this conjecture for $n\neq 2^t$ would follow from our characteristic rank conjecture (Conjecture~\ref{new-amazing-conjecture}). To explain the difference between the $k$ even and odd cases for $n=2^t$, we note that  the relation \eqref{eq:relation} in Proposition~\ref{prop:OGrdef} must be a relation between $(q_{2^t-k+1}\stb q_{2^t})$. But the relation~\eqref{eq:amazing_relation}
\[
\sum_{i\textrm{ even}}w_iq_{2^t-i}=0
\]
is not of the form \eqref{eq:relation} because the even part of the sum involves $q_{2^t-k}$, which is not in the $q_j$'s listed above. In particular Conjecture~\ref{conj:ogr-def} claims that there are no other relations involving $q_{2^t}$.

\subsection{Deficiency and torsion}
Our main conjecture connecting deficiency and torsion phenomena is now the following:
\begin{conjecture}\label{conj:deficiency}
  \[2^{\de(\OGr_k(n))+1}\op{Tor}\bigl(\H^*(\OGr_k(n);\Z)\bigr)=0\]
  or more generally,
  \[
  p^{\de_p(G,K)+1}\op{Tor}\bigl(\H^*(G/K ;\Z)\bigr)=0.
  \]
\end{conjecture}

We discuss some of the known cases in the context of Grassmannians and oriented Grassmannians. It is known, that all torsion in real partial flag manifolds is 2-torsion \cite{HudsonMatszangoszWendt}, which is also consistent with Conjecture~\ref{conj:deficiency} by Proposition~\ref{prop:Gr_def_bounds}.

By Conjecture \ref{conj:deficiency} and Proposition \ref{prop:OGr_def} we expect only 2-torsion in the integral cohomology $\H^*(\OGr_k(n);\Z)$ for
\begin{itemize}
	\item $\OGr_2(n)$; this is known by \cite{Lai1974},
	\item $\OGr_3(n)$, $n=2^{t-3}\stb 2^t$; the case of $\OGr_3(8)$ is known by a computation of \cite{Jovanovic2022}, but not known in general,
	\item $\OGr_4(n)$, $n=2^{t-3}\stb 2^t$ which is not known,
	\item $\OGr_k(2^t)$, for $k$ odd, which is not known.
\end{itemize}

Note that whenever $\de(\OGr_k(n))=1$, the conjecture does not state that there is actually 4-torsion in the cohomology, only that it could in theory appear. For instance, $\OGr_3(10)$ has deficiency $1$, on the other hand, by \cite[Theorem 6.1]{Jovanovic2022} it has no 4-torsion.

Also, one can observe that in all the cases of Theorem \ref{thm:4torsion-intro} where there is 4-torsion, the deficiency is 1.

We will make a more general conjecture on the appearance of 4-torsion for $\OGr_k(n)$ in Conjecture~\ref{conj1}, however note that Conjecture~\ref{conj:deficiency} extends beyond the scope of Grassmannians.

\section{Detecting 4-torsion in integral cohomology}
\label{sec:gysin}

In this section, we develop more precise criteria for the existence or non-existence of 4-torsion in the integral cohomology of oriented Grassmannians. Essentially, we use the Gysin sequence, combined with the fact that all torsion in the cohomology of Grassmannians is 2-torsion, to rewrite descriptions of Bockstein cohomology  for oriented Grassmannians in terms of Steenrod squares and multiplication by $w_1$ in the cohomology of Grassmannians.

\subsection{Gysin sequence with twisted integral coefficients}

Denote as before by $\mathscr{L}=\det S_0$ the determinant line bundle of the tautological bundle over $\Gr_k(n)$. The associated sphere bundle provides a double cover $\pi\colon\OGr_k(n)\to \Gr_k(n)$ which also induces a Gysin sequence with integer coefficients if one takes into account the local coefficient systems. Using the twisted Thom isomorphism, we can rewrite the long exact sequence of the pair $(\L, \L\su 0)$ as follows:
\[
\xymatrix{
  \cdots\ar[r]&
  {\rm H}^{i-1}(\Gr_k(n);\Z)\ar[r]^{e(\mathscr{L})}& 
  {\rm H}^i(\Gr_k(n);\L) \ar[r]^{\pi_\L^*}& 
  {\rm H}^i(\OGr_k(n);\mathbb{Z}) \ar[r]^{\de_{\L}}& 
  {\rm H}^{i}(\Gr_k(n);\Z)\ar[r]&\cdots 
}.
\]
Via Bockstein maps and mod 2 reductions, we can compare this integral Gysin sequence with its mod 2 counterparts from Section~\ref{sec:basics-mod2}: writing $w_1=w_1(\L)$ and $e=e(\L)$ for the first Stiefel--Whitney class and Euler class of $\L$, respectively, we get a ladder of exact sequences
\begin{equation}
  \label{eq:Gysinladder}
  \xymatrix{
      \ar[r]& {\rm H}^{i-2}(\Gr_k(n);\F_2)\ar[r]^{w_1}\ar[d]^{\be}
      &{\rm H}^{i-1}(\Gr_k(n);\F_2) \ar[r]^{\pi^*}\ar[d]^{\be_{\L}}
      &{\rm H}^{i-1}(\OGr_k(n);\F_2) \ar[r]^{\de}\ar[d]^{\be}
      & {\rm H}^{i-1}(\Gr_k(n);\F_2)\ar[r]\ar[d]^{\be}
      &\\
      \ar[r]&
      {\rm H}^{i-1}(\Gr_k(n);\Z)\ar[r]^{e}\ar[d]^{\rho}& 
      {\rm H}^i(\Gr_k(n);\mathscr{L}) \ar[r]^{\pi_{\L}^*}\ar[d]^{\rho_{\mathscr{L}}}& 
      {\rm H}^i(\OGr_k(n);\Z)\ar[r]^{\de_\L}\ar[d]^{\rho}& %\ar[r]^{\de_{\mathscr{E}}& 
      {\rm H}^{i}(\Gr_k(n);\Z)\ar[r]^{}\ar[d]^{\rho}& \\
      \ar[r]&
      {\rm H}^{i-1}(\Gr_k(n);\F_2)\ar[r]^{w_1}& 
      {\rm H}^i(\Gr_k(n);\F_2) \ar[r]^{\pi^*}& 
      {\rm H}^i(\OGr_k(n);\F_2) \ar[r]^\de& 
      {\rm H}^{i}(\Gr_k(n);\F_2)\ar[r]& 
  }
\end{equation}
The diagram commutes by naturality of pullbacks and Bockstein maps,  and by the following proposition:

\begin{proposition}
  \begin{enumerate}
  \item 
    If every torsion element in ${\rm H}^*(X;\Z(\mathscr{L}))$ is of order two, then we have the Bockstein--Wu formula:
    $$ e(\mathscr{L})\cdot\be(x)=\be_{\mathscr{L}}(w_1(\mathscr{L})\cdot x).$$
  \item If every torsion element in ${\rm H}^*(X,Y;\mathscr{L})$ is of order two, then we have the Bockstein coboundary formula for the connecting homomorphisms
    $$ \de\colon {\rm H}^i(Y;\F_2)\to {\rm H}^{i+1}(X,Y;\F_2),\qquad \de_{\mathscr{L}}\colon {\rm H}^i(Y;\mathscr{L}|_Y)\to {\rm H}^{i+1}(X,Y;\mathscr{L})$$
    $$ \de_\mathscr{L}\be(x)=\be\de(x).$$
  \end{enumerate}
\end{proposition}

\begin{proof}
For the proof, we will use the following general observation: Since every torsion element is of order 2, the reduction morphism $\rho_\L\colon{\rm H}^*(X;\Z(\L))\to {\rm H}^*(X;\F_2)$ is injective on 2-torsion. Thus, if both sides of an equality are 2-torsion, it is enough to show the mod 2 reduction of the statement. 
  
(1) Since the Bocksteins and $e(\L)$ are 2-torsion, both sides of the first equality are 2-torsion. The mod 2 reduction follows from the classical Wu formula:
  $$\Sq^1_{\L}(w_1x)=\Sq^1(w_1x)+w_1^2x=w_1\Sq^1(x)+w_1^2x+w_1^2x= w_1\Sq^1(x)$$
  
(2)  For the second equality, the right-hand side is a Bockstein, and the left hand-side is the $\de_{\mathscr{L}}$-image of a 2-torsion element, so again both sides are 2-torsion. The mod 2 reduction is the statement that, in our case, $\Sq^1$ commutes with the coboundary, cf.~Proposition~\ref{prop:steenrod-closed}:
  $$ \de\Sq^1=\Sq^1\de$$
  As an alternative formulation, we are showing the commutativity of the upper-right square in \eqref{eq:Gysinladder} by noting that the lower-right square commutes from naturality of boundary maps and the commutativity of the composition of the squares, which is $\de\Sq^1=\Sq^1\de$ from Proposition~\ref{prop:steenrod-closed}. 
\end{proof}

\begin{remark}
  There is a similar commutative diagram, with the roles of the trivial coefficients $\Z$ and the twist $\L$ switched.
\end{remark}

\subsection{A criterion for non-trivial extensions in the Gysin sequence}
\label{sec:main-criterion}

If $y$ is a $4$-torsion class in ${\rm H}^*(\OGr_k(n);\Z)$, then the diagram \eqref{eq:Gysinladder} of Gysin sequences contains an extension of the following form:
\begin{equation}
  \label{eq:4torsion-ladder}
  \xymatrix{
	\ar[r]&\Z/2\Z\bra a\ket\ar[r]^{\id}\ar[d]^{\be_\L}& 
	\Z/2\Z\bra c\ket\ar[r]^-{0}\ar[d]^{\cdot 2}& 
	\Z/2\Z\bra b\ket \ar[r]\ar[d]^{\be}& \\
	\ar[r]&\Z/2\Z\bra x\ket\ar[r]^{\cdot 2}\ar[d]^{\id}& 
	\Z/4\Z\bra y\ket\ar[r]^-{\bmod 2}\ar[d]^{\bmod 2}& 
	\Z/2\Z\bra x'\ket \ar[r]\ar[d]^{\id}& \\
	\ar[r]&\Z/2\Z\bra \rho_\L(x)\ket\ar[r]^{0}& 
	\Z/2\Z\bra \rho(y)\ket\ar[r]^{\id}& 
	\Z/2\Z\bra \rho(x')\ket\ar[r]& 
	&			
}\end{equation}
The outer two columns contain 2-torsion from $\Gr_k(n)$, and the middle column contains the nontrivial 4-torsion element in $\OGr_k(n)$. We can formalize this in the following criterion for existence of 4-torsion:

\begin{proposition}
  \label{prop:2tor_condition}
  The following statements are equivalent:
  \begin{enumerate}
  \item There is a 4-torsion class in $\H^*(\OGr_k(n);\Z)$.
  \item In $\H^*(\Gr_k(n);\F_2)$, the following inclusion is strict (i.e.\ the difference of the two sets is nonempty):
  \[
  \im(w_1\circ \rho)\se \im w_1\cap \im \Sq^1_\L
  \]
  \item In $\H^*(\Gr_k(n);\F_2)$, the following inclusion is strict:
  \[
  \im(w_1\circ \rho_\L)\se \im w_1\cap \im \Sq^1
  \]
  \end{enumerate}
\end{proposition}

\begin{proof}
	First, we show that the inclusion of (2) always holds. The inclusion $\im (w_1\circ \rho)\se \im w_1$ is immediate, so we only have to show $\im(w_1\circ \rho)\se \im \Sq^1_\L$. For all elements $z\in \im \rho$, $\Sq^1(z)=0$. Therefore $w_1\cdot z=(\Sq^1_\L+\Sq^1)(z)=\Sq^1_\L(z)$.
	The inclusion in point (3) is entirely analogous.
	
  (2)$\Rightarrow$(1): Assume that there is an element in the difference, i.e.,\ there exists an element
  \[z=\Sq^1_\L(a)=w_1\cdot z'\in \H^*(\Gr_k(n);\F_2),\]
  which is not in the image of $w_1\circ \rho$. 
	
  We claim that $\pi^*_\L\be_\L(a)$ is a nonzero, 2-divisible, 2-torsion element, i.e.\ it is 4-torsion. First, it is nonzero, otherwise $\be_\L(a)$ would be in the image of $e(\L)$, which is impossible since then $z$ would be in the image of $w_1\circ \rho$, contradicting the base assumption. It is clearly 2-torsion, since it is in the image of the Bockstein map. Finally, $z=w_1\cdot z'$ implies $\pi^*(z)=\rho\circ \pi^*_\L\circ \beta_\L(a)=0$. Therefore, the mod 2 reduction of $\pi^*_\L\be_\L(a)$ is zero, hence it is 2-divisible. 
	
  (1)$\Rightarrow$(2): By Theorem~\ref{thm:torsion-bound-flags}, we can have at worst 4-torsion, so assume ${\rm H}^*(\OGr_k(n);\Z)$ contains nontrivial 4-torsion. Then it arises as $y$ in Diagram \eqref{eq:4torsion-ladder}, together with some $a,c,b,x,x'$ as in the diagram. Then $\rho_\L(x)$ is an element in $\im w_1\cap \im \Sq^1$. We claim that it is not in $\im(w_1\circ \rho)$. Indeed, if $\rho_\L(x)=w_1\cdot \rho(v)$ for some $v\in \H^*(\Gr_k(n);\Z)$, then $e(\L)\cdot v-x\in \ker \rho_\L$, i.e.,\ it is 2-divisible. Since $x$ is 2-torsion, and $e(\L)\cdot v$ is also 2-torsion, this means that $e(\L)\cdot v=x$. But that would mean $y=\pi_\L^*x=0$, which is impossible, so $\rho_\L(x)$ is not contained in $\im(w_1\circ \rho)$.
	
  The equivalence (1)$\Leftrightarrow$(3) is obtained by repeating the proof for the other Gysin sequence, resp. the diagram obtained from \eqref{eq:Gysinladder} by switching the role of the coefficients $\Z$ and $\L$.
\end{proof}

\section{Relating Koszul homology and the kernel of \texorpdfstring{$w_1$}{the first Stiefel--Whitney class}}
\label{sec:koszul-vs-ker}

In this section we discuss the relation between two different pictures we can use to approach the cohomology of oriented Grassmannians. On the one hand, the Gysin sequence for the double covering $\OGr_k(n)\to \Gr_k(n)$ describes the cohomology ${\rm H}^*(\OGr_k(n);\F_2)$ in terms of the kernel and cokernel of the multiplication by $w_1$ on ${\rm H}^*(\Gr_k(n);\F_2)$. On the other hand, the kernel and cokernel of $w_1$ can also be described using the Koszul complex for the ideal $(q_{n-k+1},\dots,q_n)$ in $W_2=\F_2[w_2,\dots,w_k]$. The latter point was used extensively in our previous paper \cite{OGr3} to understand the cohomology of $\OGr_3(n)$. The former point of view now plays a significant role in the present paper, through our use of Stong's results \cite{Stong} in the proof of the upper-bound part of the characteristic-rank conjecture. The goal now is to connect these two pictures. 
\subsection{Recollection of notation}

The general procedure for connecting the first Koszul homology with the kernel of $w_1$ was already outlined in \cite[Section~5]{OGr3}, we recall the main points relevant for our discussion. In the following, let $k$ be fixed. 

First we fix some notation - the definition of $q_i$ and $Q_i$ is described above in \eqref{eq:HGr}, \eqref{eq:C} and \eqref{eq:qi}. Recall that $\tilde{q}_i=Q_i|_{w_1=0}\in W_1$ and $p_i\in W_1$ is the unique class satisfying $w_1p_i=Q_i+\tilde{q}_i$. These classes satisfy the following recursions
\begin{equation}\label{eq:recursion} 
  q_j=\sum_{l=2}^k w_lq_{j-l},\qquad 	p_i=Q_{i-1}+\sum_{l=2}^k  w_l p_{i-l}.
\end{equation}
	
The classes $Q_j$ and $P_j=Q_j+q_j$ also have explicit descriptions in terms of Stiefel--Whitney monomials. To formulate this description, we use the following notation. For a tuple $a=(a_1,\dots,a_k)$ we denote by $w^a=\prod_{i=1}^kw_i^{a_i}$ the corresponding Stiefel--Whitney monomial. The mod 2 multinomial coefficients corresponding to $a$ is denoted by $\binom{|a|}{a}$, where $|a|=\sum_{i=1}^ka_i$.

\begin{lemma}
  \label{lem:qj-monomial}
  \[
  Q_j=\sum_{j=a_1+2a_2+\cdots+ka_k}\binom{|a|}{a} w^a
  \]
  \[
  P_j=\sum_{j=a_1+2a_2+\cdots+ka_k, a_1\geq 1}\binom{|a|}{a} w^a
  \]
\end{lemma}

\begin{proof}
  The statement on $Q_j$ follows from the recursion $Q_j=\sum_{i=1}^k w_iQ_{j-i}$ since $\binom{|a|}{a}$ counts the number of ways the monomial $w^a$ appears in the recursion. Since $q_j$ has a similar description, 
  \[q_j=\sum_{j=2a_2+\ldots +ka_k}\binom{|a|}{a}w^a,\]
  the statement for $P_j=Q_j+q_j$ follows by taking the sum of the two expressions.
\end{proof}

\subsection{Recursions in Koszul homology and pushforwards and pullbacks}
A $W_2$-relation of the form $\sum_{j=0}^{k-1} c_j q_{n-j}=0$ gives rise to an anomalous class in $\ker w_1\se \H^*(\Gr_k(n))$, via a boundary map $\de$ in a long exact sequence of Koszul homologies, see \cite[Section~5]{OGr3}. Explicitly, the boundary is of the form
\begin{equation}\label{eq:koszul_boundary}
		\de_n\left(c_0\stb c_{k-1}\right)=\sum_{j=0}^{k-1} c_j p_{n-j}\in \ker w_1.
\end{equation}
Using the recursion \ref{eq:recursion}, one can show that a relation between $q_{n-k+1}\stb q_n$ gives a relation between $q_{n-k}\stb q_{n-1}$ and also $q_{n-k+2}\stb q_{n+1}$ -- we called these respectively \emph{descended} and \emph{ascended} relations in \cite[Section~4]{OGr3}. Let us denote respectively by $D$ and $A$ the operation of ascending and descending a relation; see \cite[Propositions 4.8, 4.11]{OGr3}:
\begin{equation}\label{eq:descending_operation}
D\left(c_0\stb c_{k-1}\right)=(c_1,c_0w_2+c_2\stb c_0w_{k-1}+c_{k-1},c_0w_k),
\end{equation}
\begin{equation}
	\label{eq:ascending_operation}
	A\left(c_0\stb c_{k-1}\right)=(c_{k-1},c_{0}w_k,c_1w_k+c_{k-1}w_2\stb c_{k-2}w_{k}+c_{k-1}w_{k-1}),
\end{equation}
We will now show that the Koszul boundary of the descended/ascended relations are the pullbacks/pushforwards of the boundary of the original relation.

\begin{proposition}\label{prop:ascending_descending_pushing_and_pulling}
Assume $k\geq 2$ and let 
\[
(c_0\stb c_{k-1})\in W_2^k 
\]
be the coefficients of a $W_2$-relation between $q_{n-k+1}\stb q_n$, i.e.,\ $\sum_{l=0}^{k-1} c_lq_{n-l}=0$. Denoting by $i_{n-1}\colon\Gr_k(n-1)\to \Gr_k(n)$ the natural inclusion, we have
\begin{equation}\label{eq:descended_pull}
\de_{n-1}(D(c_0\stb c_{k-1}))=i_{n-1}^*\de_n(c_0\stb c_{k-1})\in \H^*(\Gr_k(n-1))
\end{equation}
Similarly, for the inclusion $i_n\colon\Gr_k(n)\to \Gr_k(n+1)$, we have
\begin{equation}\label{eq:ascended_push}
\de_{n+1}(A(c_0\stb c_{k-1}))=(i_n)_!\de_n(c_0\stb c_{k-1})\in \H^*(\Gr_k(n+1)).
\end{equation}
\end{proposition}

\begin{proof}
Using the definitions \eqref{eq:koszul_boundary}, \eqref{eq:descending_operation}, equation \eqref{eq:descended_pull} states that the following equality holds in $\H^*(\Gr_k(n-1))$:
\[
c_1p_{n-1}+\sum_{j=2}^{k-1}(c_0w_{j}+c_{j})p_{n-j}+c_0w_kp_{n-k}=i_{n-1}^*\left(
\sum_{j=0}^{k-1}c_jp_{n-j}
\right).
\]
Using the recursion \eqref{eq:recursion}, $p_n=Q_{n-1}+\sum_{l=2}^k  w_l p_{n-l}$ since $i^*w_j=w_j$, the right hand side is equal to
\[
\sum_{j=1}^{k-1}c_jp_{n-j}+c_0\left(Q_{n-1}+\sum_{l=2}^k w_l p_{n-l}\right)=c_1p_{n-1}+\sum_{j=2}^{k-1}(c_0w_j+c_j)p_{n-j}+c_0w_kp_{n-k}+c_0Q_{n-1},
\]
and noting that $Q_{n-1}=0$ in $\H^*(\Gr_k(n-1))$, this proves the first equality.

The second equality \eqref{eq:ascended_push} -- using the definitions \eqref{eq:koszul_boundary}, \eqref{eq:ascending_operation} -- states that
\[
c_{k-1}p_{n+1}+c_0w_kp_n+\sum_{j=1}^{k-2}(c_jw_k+c_{k-1}w_{j+1})p_{n-j}=(i_n)_!\left(\sum_{j=0}^{k-1}c_jp_{n-j}\right).
\]
Using the recursion $p_{n+1}=Q_{n}+\sum_{l=2}^k  w_l p_{n+1-l}$, we can write the left-hand side as
\[
\begin{split}
&c_{k-1}Q_n+c_{k-1}\sum_{l=2}^kw_lp_{n+1-l}+c_0w_k p_n+\sum_{j=1}^{k-2}(c_jw_k+c_{k-1}w_{j+1})p_{n-j}\\
=&\,c_0w_kp_n+c_{k-1}Q_n+\sum_{l=1}^{k-2}(c_{k-1}w_{l+1}+c_lw_k+c_{k-1}w_{l+1})p_{n-l}+c_{k-1}w_{k}p_{n+1-k}\\
=&\,w_k\left(  \sum_{l=0}^{k-1}c_lp_{n-l}\right)+c_{k-1}Q_n
\end{split}
\]
which by Lemma~\ref{lemma:push-pull} is equal to the right-hand side, modulo the ideal $(Q_n)$. Since $k\geq 2$ by assumption, $Q_n$ is one of the relations defining $\H^*(\Gr_k(n+1))=\F_2[w_1\stb w_k]/(Q_{n-k+2}\stb Q_{n+1})$, and this concludes the proof.
\end{proof}

\subsection{The fundamental relation, revisited}
Now we return to our examination of actual $W_2$-relations between the classes $q_j$. In \cite[Theorem 4.6]{OGr3} we proved the following $W_2$-relation between $(q_{2^t-k+1}\stb q_{2^t})$: 
\begin{equation}\label{eq:fundamental}
	\sum_{i>1 \text{ odd}} w_i q_{2^t-i}=0.
\end{equation}
Applying the boundary map in the long exact sequence of Koszul homologies described in \eqref{eq:koszul_boundary} above, gives a class
\begin{equation}\label{eq:x}
	x_{2^t}=\sum_{i>1 \text{ odd}} w_ip_{2^t-i}\in W_1,
\end{equation}
which reduces to a nonzero class in $K=\ker w_1\se {\rm H}^*(\Gr_k(2^t);\F_2)$. By Proposition \ref{prop:ascending_descending_pushing_and_pulling}, by pushing forward or pulling back this class we obtain elements of $\ker w_1$ in $\Gr_k(n)$ for all $n$. In fact, we can explicitly name these elements.

\begin{proposition}
  \label{prop:KoszulStong}
  As elements of $W_1$,
  \begin{equation}
    \sum_{i\geq 0 \text{ even}} w_i p_{2^t-i}=w_1^{2^t-1}.
  \end{equation}
\end{proposition}

\begin{proof}
  The claim is equivalent to the following equality:
  \[
  \sum_{i\geq 0}w_{2i}P_{2^t-2i}=w_1^{2^t}.
  \]
  This will be easier to establish, thanks to Lemma~\ref{lem:qj-monomial}. The proof will proceed along the lines of a similar result for the $q_j$, cf.~\cite[Sections 4.1 and 4.2]{OGr3}.

  We will show the following claim, which is an analogue of \cite[Proposition 4.5]{OGr3} (but including $w_1$ in the monomials): for a sequence $a=(a_1,\dots,a_k)$ with $\sum_{i=1}^kia_i=2^t$ and $a_1>0$, we have
  \begin{equation}
    \label{eq:multinomial}
  \binom{|a|}{a}=\sum_{j=1}^{\lfloor k/2\rfloor}\binom{|a|-1}{\hat{a}_{2j}},
  \end{equation}
  unless $a=(2^t,0,\dots,0)$. 

  The key thing to note is that \cite[Lemma~4.3]{OGr3} still applies, with the same proof. Consequently, we can still make the same arguments as in the proof of \cite[Proposition~4.5]{OGr3}. If $\binom{|a|}{a}\equiv 1\bmod 2$ there is a unique index $l$ such that $\binom{|a|-1}{\hat{a}_l}\equiv 1\bmod 2$. By the same argument as in loc.cit., we find that if $\sum_{i=1}^kia_i=2^t$ then $l$ must be even unless we are in the exceptional case $a=(2^t,0,\dots,0)$. Similarly, by the same argument as in \cite[Proposition~4.5]{OGr3}, we find that for $\binom{|a|}{a}\equiv 0\bmod 2$, then for any $\binom{|a|-1}{\hat{a}_{2j}}\equiv 1\bmod 2$, there is a unique $l\neq 2j$ with $\binom{|a|-1}{\hat{a}_l}\equiv 1\bmod 2$, and $l$ is even. This establishes \eqref{eq:multinomial} and thus the claim of the proposition.
\end{proof}

\begin{proposition}
	As elements of $W_1$,
	\begin{equation}
		\sum_{i\text{ odd}}w_i Q_{2^t-i}=w_1^{2^t}
	\end{equation}
\end{proposition}
\begin{proof}
 This follows from Proposition~\ref{prop:KoszulStong} and \cite[Theorem~4.6]{OGr3}. Alternatively, it can be established with the same proof as used for these results. 
\end{proof}

As a consequence, under the translation between Koszul homology and the kernel of $w_1$, we find that the relation
\[
\sum_{i\geq 0}w_{2i}q_{2^t-2i}=0
\]
of \cite[Theorem~4.6]{OGr3} corresponds to the element
\[
\sum_{i\geq 0}w_{2i}p_{2^t-2i}=w_1^{2^t-1}
\]
in the kernel of $w_1$. In particular, this provides an independent proof of the upper-bound part $\op{ht}(w_1)\leq 2^t-1$ of Stong's theorem \cite{Stong} (first proposition of the Introduction).

\subsection{Kernel generators for \texorpdfstring{$k=3,4$}{k=3,4}}

The behaviour of anomalous classes in the low-rank cases $k=3,4$ is different from the behaviour for $k\geq 5$. In particular, the degrees of the smallest anomalous generators are different. Recall that the Koszul homology generators for $k=3,4$ arise from ascending or descending the fundamental relation $q_{2^t-3}=0$, cf.\ in particular \cite[Proposition~4.7]{OGr3}. The following results describe precisely the relation between the corresponding element $p_{2^t-3}$ in the kernel of $w_1$ and the highest non-trivial power of $w_1$ in Stong's height results \cite{Stong}. 

\begin{proposition}\label{prop:p_for_k3}
  For $k=3,4$, we have $w_1^{2^t-1}=w_3p_{2^t-3}$ in ${\rm H}^*(\Gr_k(2^t);\F_2)$.
\end{proposition}

\begin{proof}
  We first note that
  \[
  Q_{2^t-1}=\sum_{a_1+2a_2+3a_3=2^t-1}\binom{a_1+a_2+a_3}{a_1,a_2,a_3} w_1^{a_1}w_2^{a_2}w_3^{a_3},
  \]
  i.e., it consists of all the monomials $w_1^{a_1}w_2^{a_2}w_3^{a_3}$ of degree $2^t-1$ for which the exponents $a_1,a_2,a_3$ have pairwise disjoint binary expansions. On the other hand, $P_{2^t-3}$ consists of all the monomials $w_1^{a_1}w_2^{a_2}w_3^{a_3}$ of degree $2^t-3$ where the exponents $a_1,a_2,a_3$ have disjoint binary expansions, and where $a_1\geq 1$.

  We show that $Q_{2^t-1}-w_1^{2^t-1}$ is divisible by $w_3$. Assume a monomial $w_1^{a_1}w_2^{a_2}$ of degree $2^t-1$ is admissible. Then $2^t-1=a_1+2a_2$, which means that the binary expansions of $a_1$ and $2a_2$ form a disjoint decomposition of a string of 1s of length $t$. Since the binary expansion of $a_2$ removes one trailing 0, this will introduce overlap between the expansions of $a_1$ and $a_2$, unless $a_2=0$. In particular, no monomial of the form $w_1^{a_1}w_2^{a_2}$ of degree $2^t-1$ is ever admissible.

  Finally, we want to show that a monomial $w_1^{a_1}w_2^{a_2}w_3^{a_3}$ is admissible if and only if $w_1^{a_1+1}w_2^{a_2}w_3^{a_3-1}$ is admissible. We must have $a_2$ even, since otherwise $a_1$ and $a_3$ are both even, which is impossible because $a_1+2a_2+3a_3=2^t-1$. So exactly one of $a_1$ and $a_3$ is odd. In the binary expansion, we exchange the last bit of $a_1$ and $a_3$, and this doesn't affect the admissibility.

  We have thus shown that we get a bijection between monomials of $Q_{2^t-1}-w_1^{2^t-1}$ and $P_{2^t-3}$, sending a monomial $w_1^{a_1}w_2^{a_2}w_3^{a_3}$ to $w_1^{a_1+1}w_2^{a_2}w_3^{a_3-1}$. This shows that $Q_{2^t-1}=w_1^{2^t-1}+w_3p_{2^t-3}$, which shows the claim since $Q_{2^t-1}=0$ in ${\rm H}^*(\Gr_3(2^t);\F_2)$.

The argument still works for $k=4$ with minor obvious changes. The key points above are that the only odd Stiefel--Whitney classes are $w_1$ and $w_3$. 
\end{proof}

\begin{corollary}
In ${\rm H}^*(\Gr_3(2^t-1);\F_2)$ and ${\rm H}^*(\Gr_4(2^t);\F_2)$,  we have $p_{2^t-3}\in\ker w_1$. 
\end{corollary}

\begin{proof}
  We have $w_1p_{2^t-3}=P_{2^t-3}$. Moreover, $q_{2^t-3}=0$ by \cite{Fukaya} and \cite[Lemma 2.3]{Korbas2015}, cf.\ also \cite[Proposition 4.7]{OGr3}. So we have $P_{2^t-3}=Q_{2^t-3}$. Therefore, whenever $Q_{2^t-3}$ is among the relations defining ${\rm H}^*(\Gr_k(n);\F_2)$, we get $p_{2^t-3}\in\ker w_1$. The claim follows because ${\rm H}^*(\Gr_k(n);\F_2)\cong\mathbb{F}_2[w_1,\dots,w_k]/(Q_{n-k+1},\dots,Q_n)$.
\end{proof}

In particular, we can describe the generators $d_n,a_n$ of the kernel $\ker w_1\se \H^*(\Gr_3(n);\F_2)$ as follows.
\begin{proposition}
	For $\Gr_3(n)$ with $2^{t-1}<n<2^t-3$, for $t>3$ and $j:=n-2^{t-1}$ the kernel $K=\ker w_1$ is generated as a $\H^*(\Gr_3(n);\F_2)$-module by two elements:
	\[
	d_n=p_{2^t-3},\qquad a_n=w_3^j\cdot w_1^{2^{t-1}-1}
	\]
\end{proposition}
\begin{proof}
	To see that the kernel is generated by these classes $a_n$ and $d_n$ (defined by the ascending and descending operations described in \eqref{eq:descending_operation} and \eqref{eq:ascending_operation}) see \cite[Proposition 6.5]{OGr3}. For $n=2^t-1$, the kernel is generated by $p_{2^t-3}$. Using Proposition \ref{prop:ascending_descending_pushing_and_pulling} we can identify the classes with the pullbacks and pushforwards of $p_{2^t-3}$, which are respectively $p_{2^t-3}$ and $w_3^j\cdot w_1^{2^{t-1}-1}$, using Proposition \ref{prop:p_for_k3} for the latter.
\end{proof}

\subsection{A Schubert calculus proof of Stong's lemma}

In this section we give an alternate proof of Stong's height formula, using Schubert calculus. 

To compute $w_1^p$, one can apply the Pieri formula to $c_1^p$ and take the reduction of its Schubert expansion modulo 2 \cite{BorelHaefliger}. The coefficient of the Schubert class $s_\la$ in $c_1^{|\la|}$ is then given via Pieri's formula by the number of paths $p_\la$ in Young's lattice from the origin to $\la$. This in turn is given by the number of standard Young tableaux. The number of standard Young tableaux is given by the hook length formula:
\begin{equation}\label{eq:hook_length}
  p_\la=\frac{|\la|!}{\prod_{(i,j)\in \la} h_{\la}(i,j)}
\end{equation}
where $h_\la(i,j)$ denotes the hook-length of the $(i,j)$th cell of $\la$. By determining the parity of $p_\la$, we can give a Schubert calculus proof of Stong's height formula.

The modularity properties of hook-lengths have been investigated in \cite{oddPartitions}. For the convenience of the reader we recall some terminology, which is described in more detail in e.g.\ \cite{JamesKerber1981} or \cite{Olsson1993}.

A \emph{hook} in a partition $\la$ is the union of the boxes right and below a given box of the Young diagram\footnote{We will use the English notation for Young diagrams, see Figure \ref{fig:hook}.} of $\la$. The \emph{hook-length} of a hook is the number of boxes in the hook -- a hook of length $p$ is also called a $p$-hook. A partition is called a \emph{$p$-core} if it has no hooks of length $p$. For every partition $\la$, there is a unique process to obtain a subdiagram ${\rm core}_p(\la)\se \la$ called its \emph{$p$-core}, which is a $p$-core (i.e.\ it has no hooks of length $p$). We briefly describe this process.

The \emph{rim} (or border) of a partition is the union of all the right-lower-most boxes in the Young-diagram (i.e.\ boxes that do not have any box right, nor below them). A \emph{rim ($q$-)hook} (or border strip) is a contiguous strip (of length $q$) that connects two rim boxes, consisting of boxes which do not have an entry to the south-east of them (or the shortest possible). 	See Figure \ref{fig:hook}.

\begin{figure}
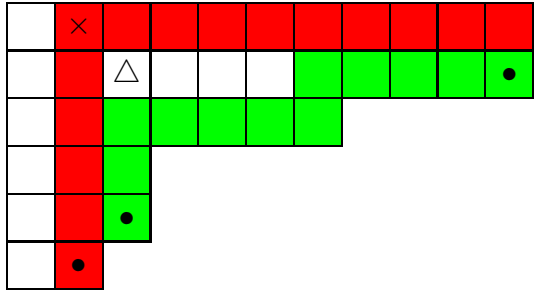
\label{fig:hook}
  \ydiagram[*(red)]{2+9,1+1,1+1,1+1,1+1}
  *[*(red) \bullet]{10+1,0,0,0,0,1+1}
  *[*(green)]{0,6+4,2+5,2+1,0,0,0}
  *[*(green) \bullet]{0,10+1,0,0,2+1,0,0}
  *[*(red) \times]{1+1,0,0,0,0}
  *[*(white) \triangle]{0,2+1,0,0,0,0}
  *[*(white)]{11,11,7,3,2,1}
  \caption{The hook of $\times$ is denoted in red, it has hook-length 15.  The rim-hook corresponding to $\triangle$ has length 12 and is marked in green. The 12-core of $\la$ is the partition obtained by removing the green rim-hook.}
\end{figure}

The $p$-core of a partition is obtained by successively removing $p$-rim hooks until there are none left. It can be shown that this process results in the same partition, which is a $p$-core, independently of the order of removing the rim hooks. Let us remark that removing a $p$-rim hook corresponding to a box is the same as removing the corresponding $p$-hook and shifting everything below the hook up and left by one in each direction. 

We will use the following description of the modularity properties of hook-lengths, \cite[Section 6]{Olsson1993}, see also \cite[Lemma 1]{oddPartitions}.
\begin{lemma}\label{lemma:oddlemma}
  Let $2^t\leq n< 2^{t+1}$ and $\la$ be a partition of $n$. 
  
  Then $p_\la$ is odd if and only if $\la$ has a unique $2^t$-hook and  $p_{\mu}$ is odd for $\mu={\rm core}_{2^t}(\la)$.
\end{lemma}
Applying this Lemma, we immediately obtain the following description of $w_1^{2^t}$.

\begin{proposition}\label{prop:w1_twopower}
  In $W_1$,
  \[
  w_1^{2^t}=\sum_{\la \text{ hook}} s_\la.
  \]
\end{proposition}

\begin{proof}
For hooks $\la$ of length $2^t$, the conditions of Lemma \ref{lemma:oddlemma} are satisfied, so $p_\la$ is odd for all such $\la$. More explicitly, if $\la=(2^t-i,1^i)$, then $p_\la=\binom{2^t-1}{i-1}$ is odd by Lucas' theorem.

Any other partition $\la$ of $2^t$ does not contain a hook of length $2^t$, so $p_\la$ is even by Lemma \ref{lemma:oddlemma}.
\end{proof}

We will need the following lemma about partitions with odd $p_\la$:
\begin{lemma}\label{lemma:oddpartitions}
	The number of standard Young tableaux $p_\la$ is odd for 
\begin{itemize}
	\item[i)] $\la=(2^t-k,k-2,1)$ for $1\leq k\leq 2^{t-1}$,
	\item[ii)] $\la =(2^t+1-k,2^t+1-k,k-1,2,1^{k-4})$ for $5\leq k\leq 2^{t-1}+1$.
\end{itemize}
\end{lemma}
\begin{proof}
  i) We show this case directly: the hook-lengths are row-by-row
  \[(1\stb 2^{t}-2k+2,2^t-2k+4\stb 2^t-k,2^t-k+2),(1\stb k-3,k-1),(1)\]
  Using the symmetry that the 2-adic valuation of $2^t-i$ is equal to the 2-adic valuation of $2^t>i>0$, this is the same set of 2-adic valuations as
  \[(1\stb 2^{t}-2k+2,2^t-2k+4\stb 2^t-k,2^t-k+2),(2^t-k+1,2^t-k+3\stb 2^t-1),(1).\]
  This contains all the even numbers between $1\stb 2^t-1$ exactly once, so its 2-adic valuation is equal to the 2-adic valuation of $(2^t)!$ and therefore $p_\la$ is odd by \eqref{eq:hook_length}.
  
  ii) We check this case using Lemma \ref{lemma:oddlemma}. Indeed, it has a unique hook of length $2^t$, so by the Lemma it is enough to show that by removing it, we obtain a partition $\mu$, for which $p_\mu$ is odd. Removing the hook, we obtain the partition $\mu=(2^t-k,k-2,1)$, which was covered in case i).
\end{proof}

Thus we obtain an alternative proof of Stong's theorem \cite{Stong} (which can be found as first proposition of the Introduction):

\begin{theorem}[Stong] \label{thm:stong}
For $5\leq k\leq n-5$, $2^{t-1}<n<2^t$,
\[{\rm ht}(w_1)=2^{t}-1.\]  
\end{theorem}

\begin{proof}
By Proposition \ref{prop:w1_twopower}, $w_1^{2^t}$ is equal to the sum $\sum s_\la$, where $\la$ runs through hooks of length $2^t$. However, the maximal hook that fits into the $k\times (2^t-k)$ rectangle has length $2^t-1$, so $w_1^{2^t}=0$ in $\H^*(\Gr_k(2^t);\F_2)$ which gives the bound $\het(w_1)\leq 2^t-1$.

It remains to show that $w_1^{2^t-1}$ is nonzero. By the duality $\Gr_k(n)\iso \Gr_{n-k}(n)$, we can assume that $k\leq 2^{t-1}$. Moreover, it is enough to show the Proposition for $n=2^{t}+1$, since for $i\colon\Gr_k(n)\to \Gr_k(n+1)$ we have $i^*w_1=w_1$.  

So it is enough to give a partition $\la\se k\times(2^{t}+1-k)$ of $|\la|=2^{t}-1$, such that $p_\la$ is odd. The partition in part ii) of Lemma~\ref{lemma:oddpartitions} is such a partition, which allows us to conclude.
\end{proof}

\begin{remark}
  It would be possible to show part i) of Lemma \ref{lemma:oddpartitions} using Lemma \ref{lemma:oddlemma}, but there are some case distinctions that can be avoided this way. In particular, the recursion on the core of the partition travels through different paths in the tree described in  \cite{oddPartitions}. Similarly, it would be possible to show part ii) of Lemma \ref{lemma:oddpartitions} directly, but the computation is somewhat more complicated.
\end{remark}

\section{Partial results on the characteristic rank conjecture}
\label{sec:charrank}

In the following section, we revisit the characteristic rank conjecture~\ref{new-amazing-conjecture}. Using Stong's formula for the height of $w_1$, cf. \cite{Stong} or Theorem~\ref{thm:stong} above, we prove the upper bound on the characteristic rank in Conjecture~\ref{new-amazing-conjecture}. We also prove the characteristic rank conjecture in the cases $k=5$, $n=2^t-1,2^t$ and $k=6$, $n=2^t$. The 4-torsion examples will be deduced from this in Section~\ref{sec:4torsion-examples}.

Recall from \cite{OGr3} or Section~\ref{sec:basics-mod2} that there are two ways to think about the characteristic rank or the anomalous classes. On the one hand, the boundary map ${\rm H}^*(\OGr_k(n);\F_2)\to {\rm H}^*(\Gr_k(n);\F_2)$ of the Gysin sequence maps the anomalous classes to $\ker w_1\subset{\rm H}^*(\Gr_k(n);\F_2)$. This viewpoint will be used below to establish the upper bound part of the characteristic rank conjecture from Stong's height formula. On the other hand, the Koszul homology picture of \cite{OGr3} relates the anomalous classes via the Koszul boundary to relations between the $q_{n-k+1},\dots,q_n$ in the definition of $C=\F_2[w_2,\dots,w_k]/(q_{n-k+1},\dots,q_n)$. This viewpoint will be used below to establish the cases of the characteristic rank conjecture mentioned above, via a brute force analysis of the monomials appearing in $q_{n-k+1},\dots,q_n$ and their possible cancellations in a relation of small degree. The precise translation between these viewpoints is discussed in Section~\ref{sec:koszul-vs-ker}. 

\subsection{Height and characteristic rank}

Recall that the \emph{height} of an element $x\in R$ in a ring is the maximal power for which it is nonzero\footnote{Note that the rings we consider here are actually graded finite-dimensional algebras over a field, so there always is an $i$ such $x^i=0$ for positive-degree elements $x$.}:
\[\het x=\sup \{i\mid x^i\neq 0\}.\]

The height of $w_1$ in ${\rm H^*}(\Gr_k(n))$ gives a trivial upper bound for the characteristic rank of the tautological bundle $S\to \OGr_k(n)$, cf.\ Definition~\ref{def:charrank} and the discussion after it. More generally, the characteristic rank of a vector bundle can be bounded by the height of its top Stiefel--Whitney class, as formulated in the following simple proposition:

\begin{proposition}
  \label{prop:crk-sphere-bdl}\,%Hack
  \begin{enumerate}
  \item 
    Let $E\to X$ be a real vector bundle of rank $n$ over a manifold, with $w_n(E)\neq 0$ and assume that the characteristic classes $w_i(E)$ generate the mod 2 cohomology ${\rm H}^*(X;\mathbb{F}_2)$. Denote by $F:=E\setminus z(E)$ the complement of the zero section.\footnote{Alternatively, if $E$ has a metric, we can use the sphere bundle of $E$.} Then 
    \[n\het(w_n(E))\geq \crk(\pi_F^* E\to F)+2-n\]
  \item Let $E\to X$ is a real vector bundle of rank $n$ over a manifold, with $w_1(E)\neq 0$ and assume that the characteristic classes $w_i(E)$ generate the mod 2 cohomology ${\rm H}^*(X;\mathbb{F}_2)$. Denote by $F:=\det E\setminus z(\det E)$ the complement of the zero section \emph{of the determinant bundle}. Then 
    \[\het(w_1(E))\geq \crk(\pi_F^* E\to F)+1\]
  \end{enumerate}
\end{proposition}

\begin{proof}
  (1) 
  Consider the Gysin sequence associated to $F$:
  \[
  \cdots \to {\rm H}^{i-1}(X;\F_2)\to {\rm H}^{i-1}(F;\F_2)\xrightarrow{\delta} {\rm H}^{i-n}(X;\F_2) \xrightarrow{w_n(E)} {\rm H}^i(X;\F_2)\to\cdots
  \]
  By assumption the characteristic classes of $E$ generate ${\rm H}^*(X;\F_2)$, so classes which are not characteristic classes are detected by $\ker w_n(E)$ under the boundary map $\delta$. In particular, the lowest nonzero degree of $\ker w_n(E)$ (plus $n-1$, the degree of $\delta$) is equal to $\crk(\pi_F^*E\to F)+1$. Set $h=\het(w_n(E))$, which is finite since $X$ is a manifold. Then by definition, $w_n(E)^h$ is a nonzero element of degree $nh$ in $\ker w_n(E)$.

  (2) The argument is the same, using the Gysin sequence for the determinant bundle instead.
\end{proof}

\begin{theorem}
  \label{thm:stong-upper-bound}
  Let   $5\leq k\leq 2^{t-1}<n\leq 2^t$ and $t\geq 5$. Then for the tautological bundle $S\to \OGr_k(n)$, the characteristic rank is at most
  \[  	\crk(S)\leq \min (2^t-2,k(n-2^{t-1})+2^{t-1}-2).\]
\end{theorem}

\begin{proof}
  Let $S_0\to \Gr_k(n)$ be the tautological bundle; it is well-known that its characteristic classes generate ${\rm H}^*(\Gr_k(n))$. Since $\OGr_k(n)$ is the sphere bundle of $\det S_0$ and $w_1(\det S_0)=w_1(S_0)$ and $S\to\OGr_k(n)$ is the pullback of $S_0$ along $\pi\colon\OGr_k(n)\to \Gr_k(n)$, we can apply part (2) of Proposition~\ref{prop:crk-sphere-bdl}, and obtain the upper bound 
  \[\crk(S)\leq \het w_1(S_0)-1=2^t-2\]
  where the last equality is a theorem of Stong~\cite{Stong}, cf.\ Theorem~\ref{thm:stong}. 
  
  To obtain the other upper bound, recall from Lemma~\ref{lemma:push-pull} that if $i\colon\Gr_k(m)\se \Gr_k(n)$ is the canonical inclusion induced by the inclusion of vector spaces $\mathbb{R}^m\se\mathbb{R}^n$, then $\ker i_!=(0)$. Also, if $x\in \ker w_1\se {\rm H}^*(\Gr_k(m))$, then $i_!x\in \ker w_1 \se {\rm H}^*(\Gr_k(n))$, by the projection formula (since $w_1$ on $\Gr_k(m)$ is the restriction of $w_1$ on $\Gr_k(n)$). For $m=2^{t-1}$, we can use Lemma~\ref{lemma:push-pull} to see that
  \[
  i_! w_1^{2^t-1}=w_k^{n-2^{t-1}}w_1^{2^t-1}
  \]
  is a nonzero element in $\ker w_1$, giving the other upper bound.
\end{proof}

\subsection{A system of anomalous classes}

We briefly discuss some anomalous classes that are in the degree above the expected characteristic rank. Viewing anomalous classes as coming from $\ker w_1\subset{\rm H}^*(\Gr_k(n);\F_2)$, we can use Stong's theorem, cf. \cite{Stong} or Theorem~\ref{thm:stong}, to write down classes annihilated by $w_1$ which (under the characteristic rank conjecture) correspond to the anomalous generators of smallest degree. 

By Stong's theorem, $w_1^{2^{t}-1}$ is a nonzero element in $\ker w_1\se {\rm H}^*(\Gr_k(n))$ for any $k\geq 4$ and any $n$ with $2^{t-1}<n\leq 2^t$. We call 
\begin{equation}\label{eq:dn}
d_n=w_1^{2^t-1}
\end{equation}
the \emph{descended class $d_n$}.

If $i\colon\Gr_k(2^{t})\to \Gr_k(n)$ for $n=2^t+j$, the pushforwards of $w_1^{2^t-1}$ are also nonzero classes in $\ker w_1\se {\rm H}^*(\Gr_k(n))$; let the \emph{ascended class $a_n$} be
\begin{equation}\label{eq:an}
a_n=i_!w_1^{2^t-1}=w_1^{2^t-1}w_k^j,
\end{equation}
where the last equality holds by Lemma \ref{lemma:push-pull} describing the Gysin map $i_!$.

\begin{remark}
	As we described in Propositions \ref{prop:ascending_descending_pushing_and_pulling} and \ref{prop:KoszulStong}, the connection to the Koszul homology picture can be made explicit.
\end{remark}

\begin{remark}
  The classes $a_n$ and $d_n$ behave differently. 
  By Stong's theorem, we have $i^*w_1^{2^{t+1}-1}=0$ for $i\colon\Gr_k(2^{t})\to \Gr_k(2^{t+1})$. On the other hand, for $j\colon\Gr_k(2^t)\to \Gr_k(n)$ for $n>2^t$, $j_! w_1^{2^t-1}$ is never 0. However, for large enough $n$, $a_n$ can be expressed in terms of other elements in the kernel. In this case, its lift to the oriented Grassmannian will not be a $C$-module generator (in any minimal presentation).
\end{remark}

\subsection{Some cases of the characteristic rank conjecture}
\label{subsec:charrank}

We now want to establish the characteristic rank conjecture for $k=5$ and $n=2^t-1,2^t$. As a warm-up, we start with an easy case of the characteristic rank conjecture.

\begin{theorem}
  \label{thm:charrank-2t}
  Let $n=2^t$, $t\geq 4$. Then the class $d_{2^t}\in {\rm H}^{2^t-1}(\Gr_5(2^t);\F_2)$ is a nonzero class and
  \begin{equation}
    \crk(S\to\OGr_5(2^t))=2^t-2.
  \end{equation}
\end{theorem}

\begin{proof}
  We already know that $d_{2^t}$ is nonzero from \cite[Theorem 5.7]{OGr3} or via the identification with $w_1^{2^{t}-1}$ in Proposition \ref{prop:KoszulStong}. It suffices to show that there are no anomalous classes in degrees below $2^t-1$. In the Koszul homology picture of \cite{OGr3}, this means that there are no relations between $q_{2^t-4},\dots,q_{2^t}$ in degrees $\leq 2^t-1$. By the lemma of Korba\v s in \cite{Korbas2015}, we know that $q_{2^t-4}$ and $q_{2^t-3}$ are non-zero, so there cannot be relations in these degrees. We discuss the two remaining degrees below. 
  
(degree $2^t-2$) We want to show that $q_{2^t-2}$ and $w_2q_{2^t-4}$ are independent. Since $2^t-8$ is divisible by 8, the equation $2a+4=2^t-4$ has an even solution and therefore $w_2^aw_4$ with $a$ even appears as monomial in $q_{2^t-4}$. But the monomial $w_2^{a+1}w_4$ cannot appear in $q_{2^t-2}$ and thus cannot be cancelled, showing the independence.
  
(degree $2^t-1$) We want to show that $q_{2^t-1}$, $w_2q_{2^t-3}$ and $w_3q_{2^t-4}$ are independent. Since $2^t-8$ is divisible by $8$, the equation $4a+5=2^t-3$ has an even solution and therefore $w_4^aw_5$ with $a$ even appears as a monomial in $q_{2^t-3}$. But $w_2w_4^aw_5$ cannot appear in $q_{2^t-1}$ because of two odd exponents, and obviously cannot appear in $w_3q_{2^t-4}$ as well. On the other hand, $2^t-4$ is exactly divisible by $4$ and therefore $w_4^a$ appears with odd exponent $a$. In particular, $w_3w_4^a$ cannot appear in $q_{2^t-1}$ and thus cannot cancel. So there cannot be any nontrivial relation.
\end{proof}

Now we will establish a case of the characteristic rank conjecture which will be relevant for our 4-torsion examples. 

\begin{theorem}\label{thm:charrank}
  Let $n=2^t-1$, $t\geq 4$. Then
  \begin{equation}
    \crk(S\to\OGr_5(n))=2^t-2.
  \end{equation}
\end{theorem}
\begin{proof}  
	The characteristic rank is $\leq 2^t-2$ by Theorem \ref{thm:stong-upper-bound}.
	
  To prove that the characteristic rank is $\geq 2^t-2$, we need to show that there are no relations between $q_{2^t-5},\dots, q_{2^t-1}$ in degrees $\leq 2^t-1$. 
  Again, by Korba\v s' lemma in \cite{Korbas2015}, none of these elements are zero, so there are no relations in degrees $2^t-5$ and $2^t-4$. Showing that there are no relations in degrees $2^t-3,\dots,2^t-1$ is a lengthy case distinction carried out below.

  (degree $2^t-3$)  We want to show that $w_2 q_{2^t-5}$ and $q_{2^t-3}$ are independent. For this it suffices to find a monomial in $q_{2^t-3}$ which is not divisible by $w_2$. We note that for $t\geq 4$ the number $2^t-8$ is divisible by $8$, and therefore the solution to $4a+5=2^t-3$ is even. Then the monomial $w_4^aw_5$ appears in $q_{2^t-3}$.
  % alternative for odd $t$ case: If $t\equiv 1\bmod 2$, then $2^t\equiv 2\bmod 6$ and $2^t-5\equiv 3\bmod 6$. Therefore, $2^t-5$ is divisible by 3, and the result is an odd number, call it $a$. Then the monomial $w_3^{a-1}w_5$ appears in $q_{m-3}$.

(degree $2^t-2$)  We want to show that $w_3q_{2^t-5}, w_2q_{2^t-4}$ and $q_{2^t-2}$ are independent. We first want to show that $q_{2^t-2}$ doesn't appear in a relation, i.e., we want to show that $q_{2^t-2}$ contains a monomial $w_4^aw_5^2$ with $a-1\equiv 0\bmod 4$ which consequently cannot be cancelled by monomials from $w_3q_{2^t-5}$ or $w_2q_{2^t-4}$. We first note
  \[
  a=\frac{2^t-12}{4}=2^{t-2}-3,
  \]
  and consequently $a-1$ is divisible by $4$. Therefore the monomial $w_4^aw_5^2$ appears in $q_{2^t-2}$.
  % alternative for $t\equiv 0\bmod 2$: we want to show that $q_{2^t-5}$ contains a monomial $w_3^aw_5$ with $a$ even, and therefore the monomial $w_3^{a+1}w_5$ will not appear in $q_{2^t-2}$. To show the former assertion, we note that $t\equiv 0\bmod 2$ implies $2^t-10$ is divisible by 3, with the result being an even number. In particular, $3a+5=2^t-5$ has an even solution, and the monomial $w_3^aw_5$ appears in $q_{2^t-5}$.

To show that $w_3q_{2^t-5}$ and $w_2q_{2^t-4}$ are independent, we note that there is a monomial $w_2^aw_3$ in $q_{2^t-5}$ because $2^t-8$ is divisible by 4, i.e., $2a+3=2^t-5$ has an even solution. Moreover, $a=2^{t-1}-4$ is exactly divisible by $4$. Thus the binary expansion of $a-1$ ends with the two digits $11$ and therefore $a-1$ and $2$ don't have disjoint binary expansions. This means that the monomial $w_2^{a-1}w_3^2$ cannot appear in $q_{2^t-4}$ and consequently the monomial $w_2^aw_3^2$ in $w_3q_{2^t-5}$ cannot be cancelled, showing the independence.
  
(degree $2^t-1$)  Finally, we want to show the independence of $q_{2^t-1}$, $w_2q_{2^t-3}$, $w_3q_{2^t-4}$, $w_4q_{2^t-5}$ and $w_2^2q_{2^t-5}$.

We first show that $w_2q_{2^t-3}$ cannot appear in any relation. For this, we claim that there is a monomial $w_2^aw_5$ in $q_{2^t-3}$ with $a$ even. Since $w_2^{a+1}w_5$ cannot occur in $q_{2^t-1}$ and $w_2^{a-1}w_5$ cannot occur in $q_{2^t-5}$, such a monomial cannot be cancelled. To see that the monomial actually occurs in $q_{2^t-3}$, we note that $2^t-8$ is divisible by 8, and thus $2a+5=2^t-3$ has an even solution. This means that $w_2q_{2^t-3}$ appears trivially.

Now we want to show that $w_3q_{2^t-4}$ cannot appear in any relation. We claim that there is a monomial $w_2^aw_5^2$ in $q_{2^t-4}$ with $a\equiv 1\bmod 8$. This follows since $2a+10=2^t-4$ has a solution $a\equiv 1\bmod 8$ and thus the binary expansions of $a$ and $2$ are disjoint. Because $a$ is odd, the monomial $w_2^aw_3w_5^2$ cannot appear in $q_{2^t-1}$, and the monomial $w_2^{a-2}w_3w_5^2$ cannot appear in $q_{2^t-5}$. So there is no possibility to cancel, and hence $w_3q_{2^t-4}$ cannot appear in any relation.

To exclude $w_4q_{2^t-5}$ from any relation, we note that for $t\equiv 1\bmod 2$ we have $2^t-5\equiv 0\bmod 3$ and therefore $w_3^a$ with $a$ odd appears as monomial in $q_{2^t-5}$. Because $a$ is odd, $w_3^aw_4$ cannot appear in $q_{2^t-1}$ or $w_2^2q_{2^t-5}$ and thus cannot cancel. For $t\equiv 0\bmod 2$, we note that $3+4a=2^t-5$ has a solution with $a$ even and therefore $w_3w_4^a$ appears as monomial in $q_{2^t-5}$. But $w_3w_4^{a+1}$ cannot appear in $q_{2^t-1}$ and hence cannot cancel. This excludes $w_4q_{2^t-5}$ from any relation.

We are left to show that $q_{2^t-1}+w_2^2q_{2^t-5}$ is nonzero. We claim that a monomial $w_3w_4^a$ with $a$ even is contained in $q_{2^t-5}$. This follows since $2^t-8$ is divisible by $8$ and thus $3+4a=2^t-5$ has an even solution. But the monomial $w_2^2w_3w_4^a$ with $a$ even cannot appear in $q_{2^t-1}$. This concludes the proof.
\end{proof}

\begin{corollary}
  \label{cor:charrank2t}
  Let $n=2^t$. Then the class $d_n\in{\rm H}^{2^t-1}(\Gr_6(n);\F_2)$ is a nonzero class and
  \[
  {\rm crk}(S\to\OGr_6(n))=2^t-2.
  \]
\end{corollary}

\begin{proof}
  The statement $d_{2^t-1}\neq 0$ has been proved in \cite[Theorem 5.7]{OGr3}, which implies $\crk\leq 2^t-2$.

  To prove that the characteristic rank is $\geq 2^t-2$, we need to show that there are no relations between $q_{2^t-5},\dots,q_{2^t}$ in degrees $\leq 2^t-1$. Assume there is a nontrivial such relation
  \[
  \sum_{i=2^t-5}^{2^t} \lambda_i q_i=0.
  \]
  We consider the reduction modulo $w_6$ map $\mathbb{F}_2[w_2,\dots,w_6]\to \mathbb{F}_2[w_2,\dots,w_5]$ which will send the $q_j$ for $k=6$ to the $q_j$ for $k=5$. In particular, the reduction of $q_j$ modulo $w_6$ will be nonzero. Since we are looking at relations between $q_{2^t-5},\dots,q_{2^t}$ in degrees $\leq 2^t-1$, the coefficients $\lambda_i$ will be of degree $\leq 4$ and so their reductions modulo $w_6$ will also be nontrivial. Therefore, the reduction of the relation for $k=6$ will be a relation between $q_{2^t-5},\dots,q_{2^t-1}$ of degree $\leq 2^t-1$ for $k=5$ and $n=2^t-1$. By Theorem~\ref{thm:charrank}, this relation has to be 0. We get a contradiction, proving the claim.
\end{proof}

\section{An infinite family of 4-torsion classes}
\label{sec:4torsion-examples}

In this section, we discuss the existence of 4-torsion in the integral cohomology of oriented Grassmannians. Essentially, we show that for $n\neq 2^t$ the classes $a_n$ and $d_n$ from Section~\ref{sec:charrank} are reductions of 2-torsion classes. Then we show in Theorem \ref{thm:main_general} that in some cases, assuming the characteristic rank conjecture holds, one of them satisfies the criterion of Proposition~\ref{prop:2tor_condition}, implying the existence of a nontrivial 4-torsion class. In particular, since the characteristic rank conjecture holds for $\OGr_5(2^t-1)$ by Theorem \ref{thm:charrank}, this implies the existence of an infinite family of 4-torsion classes as $t$ varies. In general, based on the characteristic rank conjecture and Theorem \ref{thm:main_general}, we expect that the appearance of 4-torsion classes in the integral cohomology of oriented Grassmannians is not a sporadic phenomenon, but a typical one.

\subsection{$a_n,d_n$ are reductions of 2-torsion classes}

Recall the action of $\Sq^1$ on $W_1$:
\begin{equation}
  \label{eq:Sq1w}
  \Sq^1w_{2i}=w_1w_{2i}+w_{2i+1},\qquad \Sq^1w_{2i+1}=w_1w_{2i+1}.
\end{equation}
The twisted Steenrod squares are obtained from this via $\Sq^1_\L(x)=\Sq^1(x)+w_1x$. 

The following proposition shows that if $n\neq 2^t$, the classes $a_n$ and $d_n$ are reductions of integral classes. For this, recall from Section~\ref{sec:basics-mod2} that the Bockstein sequence
  \[
  \cdots\xrightarrow{2}{\rm H}^*(\Gr_k(n);\Z)\xrightarrow{\rho}{\rm H}^*(\Gr_k(n);\F_2)\xrightarrow{\beta}{\rm H}^{*+1}(\Gr_k(n);\Z)\to\cdots
  \]
  implies that an element $x\in{\rm H}^*(\Gr_k(n);\F_2)$ has an integral lift if and only if $\beta(x)=0$. However, since by Ehresmann's result \cite{Ehresmann} all torsion in ${\rm H}^*(\Gr_k(n);\Z)$ is 2-torsion, \cite[Lemma 2.2]{brown:bson} implies that $\beta(x)=0$ if and only if $\Sq^1(x)=0$. Similarly, a class $x\in{\rm H}^*(\Gr_k(n);\F_2)$ lifts to twisted integral cohomology if and only if $\Sq^1_\L(x)=0$. We check the Steenrod square condition:

\begin{proposition}
  \label{prop:w1_in_kerSq1}
  If $2^{t-1}<n\leq 2^t$, then  $a_n,d_n\in \ker \Sq^1\cap \ker \Sq^1_\L \subset  {\rm H}^*(\Gr_k(n);\F_2)$. Explicitly, 
  \begin{equation}
    \Sq^1w_1^{2^t-1}=\Sq^1_\L w_1^{2^t-1}=0
  \end{equation}
  and if $j=n-2^{t-1}$, then
  \begin{equation}
    \Sq^1 \left(w_1^{2^{t-1}-1}w_k^j\right)=\Sq^1_\L \left(w_1^{2^{t-1}-1}w_k^j\right)=0.
  \end{equation}
\end{proposition}

\begin{proof}
  Since $\Sq^1 w_1^{2^t-1}=w_1^{2^t}$, this is 0 by Stong's theorem, and $\Sq^1_\L w_1^{2^t-1}=w_1^{2^t}+w_1^{2^t}=0$.
  
  For the second equality, 
  \[
    \Sq^1 \left(w_1^{2^{t-1}-1}w_k^j\right)=\underbrace{\Sq^1\left(w_1^{2^{t-1}-1}\right)}_0w_k^j+w_1^{2^{t-1}-1}\Sq^1\left(w_k^j\right)=0
  \]
  Here, we use the derivation property, Stong's theorem, and the following case distinction to show the vanishing of the second summand. If $j$ is even, the claim follows from the derivation property. If $j$ is odd, $\Sq^1\left(w_k^j\right)=w_k^{j-1}\Sq^1(w_k)=w_1w_k^j$, since $w_{k+1}=0$ in ${\rm H}^*(\Gr_k(n);\F_2)$. But then the second summand vanishes because $w_1^{2^{t-1}}=0$. From this, we also get
  \[\Sq^1_\L w_1^{2^{t-1}-1}w_k^j=w_1 a_n=0,\]
  since $a_n\in \ker w_1$.
\end{proof}

Now we can establish relevant cases in which the classes $a_n$ and $d_n$ are reductions of 2-torsion classes in integral cohomology by showing they are contained in the image of $\Sq^1$. 

\begin{proposition}
  \label{prop:w1_in_imSq1} 
  The following assertions hold:
  \begin{enumerate}
  \item For any $n\neq 2^t$, $d_n\in \im \Sq^1$
  \item Assume that $k$ is even. Then $a_n\in \im \Sq^1$ if $n$ is even, and $a_n\in \im\Sq^1_\L$ if $n$ is odd.
  \end{enumerate}
\end{proposition}

\begin{proof}
  For the proof, we will use the following fact concerning the degrees of nonzero rational classes in ${\rm H}^*(\Gr_k(n);\Q)$:
  \begin{itemize}
  \item[(*)] There are no rational classes in odd degree of either twist unless $k$ is odd and $n$ is even.
  \end{itemize}

  For the assertion concerning $d_n$, let $n=2^t-1$. By Proposition \ref{prop:w1_in_kerSq1}, $d_n$ is the reduction of an integral (untwisted) class. Since $n=2^t-1$ is odd, by (*) there are no nontorsion classes in degree $n$ and so $d_n$ is the reduction of a 2-torsion class. Therefore $d_{2^t-1}$ is in the image of $\Sq^1$, i.e., $d_{2^t-1}=\Sq^1 \de_{2^t-1}$ for some $\de_{2^t-1}$. For general $n$, set $\de_{n}=i^*\de_{2^t-1}$ for the natural inclusion $i\colon\Gr_k(n)\to\Gr_k(2^t-1)$. By naturality, we have
  \[\Sq^1\de_{n}=\Sq^1i^*\de_{2^t-1}=i^*\Sq^1\de_{2^t-1}=i^*d_{2^t-1}=d_{n}\]
  which proves that $d_n\in \im \Sq^1$.
  
  Next, assume that $k$ is even. We want to show that depending on the parity of $n$,
  \begin{equation}\label{eq:an_in_imSq1}
    a_{n}\in \im \Sq^1,\qquad \text{or} \qquad 	 a_{n}\in \im \Sq^1_\L.
  \end{equation}
  We first show for $n=2^t$ that $a_n\in\im\Sq^1$. The rest of the cases is then established by induction. 
  
  Since $k$ is even, $\deg a_{2^t}=2^t-1$ is odd. By (*), there are no odd-degree cohomology classes in ${\rm H}^*(\Gr_k(n);\mathbb{Q})$, therefore $a_n$ is the reduction of some 2-torsion class from the trivial twist, i.e., $a_n=\Sq^1 \al_n$ for some $\al_n$. This proves \eqref{eq:an_in_imSq1} for $n=2^t$.

  To prove the rest of the cases, % $2^t+1<n<2^{t+1}$,
  set $\al_n=i_!\al_{n-1}$ for $i\colon\Gr_k(n-1)\to \Gr_k(n)$. There are two cases. If by induction $a_{n-1}=\Sq^1_\L\al_{n-1}$, then by Lemma \ref{lemma:pushSq1}:
  \[\Sq^1\al_n=\Sq^1i_!\al_{n-1}=i_!\Sq^1_\L \al_{n-1}=i_!a_{n-1}=a_n.\]
  If by induction $a_{n-1}=\Sq^1\al_{n-1}$, then
  \[
  \Sq^1_\L \al_n=\Sq^1_\L i_!\al_{n-1}=i_!\Sq^1 \al_{n-1}=i_!a_{n-1}=a_n.
  \]
  This concludes the proof.
\end{proof}

\begin{remark}
  This proof is an existence proof, and $\al_n$, $\de_n$ are not  explicitly determined classes. Let us note that $w_1^{2^t-1}\in W_1$ is not in the image of $\Sq^1$ -- it is only in the image of $\Sq^1$ after taking the quotient by the ideal $(Q_{n-k+1}\stb Q_{n-k})$.
\end{remark}

The following result discusses the case where $k$ is odd and $n=2^t$. It explains why excluding the case $n=2^t$ is necessary whenever $k$ is odd, cf.\ point (1) of Proposition~\ref{prop:w1_in_imSq1}.

\begin{proposition}
  \label{prop:an_oddk}
  For odd $k$, we have $w_1^{2^t-1}\not\in \im \Sq^1$ in ${\rm H}^*(\Gr_k(2^t);\F_2)$. 
\end{proposition}

\begin{proof}
	We claim that $p_\la$ defined in \eqref{eq:hook_length} is odd for the partition $\la=(2^t-k, 1^{k-1})$. Indeed, $p_\la=\binom{2^t-2}{k-1}$, which is even iff $k$ is even, by Lucas' theorem. So for $k$ odd, the coefficient of $s_\la$ in $w_1^{2^{t-1}}$ is nonzero.
  For $k$ odd, this class is the reduction of a nonzero rational class; indeed, the Schubert variety $\si_{(n-k,1^{k-1})}\se \Gr_k(n)$ is a smaller Grassmannian $\Gr_{k-1}(n-2)$ for an appropriate flag, cf.\ \cite[Proposition 5.8]{WendtChowWitt} and \cite[Lemma 5.3]{Matszangosz}. Since this class appears with zero coefficient in $\Sq^1\si_\la$ for all $\la$, $w_1^{2^t-1}\not\in \im \Sq^1$.
\end{proof}

\begin{remark}
  The proposition also provides some explanation for our expectation that there should be no 4-torsion in $\OGr_k(n)$ for $k$ odd and $n=2^t$, as formulated in Conjecture~\ref{conj1} below. In this situation, the generator of $\ker w_1\subset{\rm H}^*(\Gr_k(n);\F_2)$ fails to be in the image of $\Sq^1$, because it relates to the reduction of the fundamental class of a submanifold which itself is a smaller Grassmannian. Therefore $a_n$ is the reduction of a non-torsion class. Similar statements then hold (by pushing forward) for the ascended generators if $k$ is odd. 
\end{remark}

\subsection{An infinite family of 4-torsion classes}

We can now discuss the occurrence of 4-torsion in the integral singular cohomology of oriented Grassmannians. We first formulate our conjecture on where we should find 4-torsion classes, and where we shouldn't. 

\begin{conjecture}\label{conj1}\,%Hack: some text needed otherwise reference is broken.
  \begin{enumerate}
  \item For $k\leq 4$, all torsion in the integral cohomology of $\OGr_k(n)$ is 2-torsion.
   \item 
  For $1<k<2^t-1$, $k$ odd,  all torsion in the integral cohomology of $\OGr_k(2^t)$ is 2-torsion.
  \item If $k\geq 6$ is even, assume $k\leq n-k$ and  set
    \[
    c=\min(\deg a_n,\deg d_n)-1=\min(k(n-2^{t-1})+2^{t-1}-2,2^t-2).
    \]
    Then there is a 4-torsion class in ${\rm H}^{c+1}(\OGr_k(n);\Z)$.
  \item
    If $k\geq 5$ is odd, then let $n,t$ be such that $5\leq k\leq 2^{t-1}<n< 2^t$ and $t\geq 5$, and assume that \[2^t-1=\deg d_n<\deg a_n=k(n-2^{t-1})+2^{t-1}-1,\]
    i.e., that $n>\frac{k+1}{k}2^{t-1}$.
     Then there is a 4-torsion class in ${\rm H}^{2^t-1}(\OGr_k(n);\Z)$.

  \end{enumerate}
\end{conjecture}
In the cases not covered by the conjecture, we also expect the appearance of $4$-torsion classes in general, but we do not have explicit candidates for such classes. We have seen that parts (1) and (2) of Conjecture~\ref{conj1} are consequences of the deficiency conjecture~\ref{conj:deficiency} by Proposition \ref{prop:OGr_def}. The following theorem shows that part (3) and (4) of Conjecture~\ref{conj1} are a consequence of the characteristic rank conjecture~\ref{new-amazing-conjecture}. In particular, in its proof, we name explicit classes which form a 4-torsion extension using Proposition \ref{prop:2tor_condition}.

\begin{theorem}
  \label{thm:main_general}
  Let $5\leq k\leq n-5$, and assume that the characteristic rank conjecture~\ref{new-amazing-conjecture} holds for the given $k$ and $n$, i.e., the characteristic rank for $\OGr_k(n)$ is equal to
  \begin{equation}\label{eq:crk}
    c:=\crk(S\to \OGr_k(n))=\min(\deg d_n, \deg a_n)-1	
  \end{equation}
  Then Conjecture~\ref{conj1} holds, i.e., under the conditions in items (3) or (4) of the conjecture, we have a 4-torsion class in ${\rm H}^{c+1}(\OGr_k(n);\Z)$.
\end{theorem}

\begin{proof}
  Assuming the characteristic rank conjecture \eqref{eq:crk} and the conditions of Conjecture~\ref{conj1}, we will show that the smaller-degree class among $a_n$ and $d_n$ satisfies one of the conditions of Proposition \ref{prop:2tor_condition}. Explicitly, we show that the smaller-degree element is in $\im w_1$ and $\im \Sq^1$ (or $\im \Sq^1_\L$), but it is not in the image of $w_1\circ \rho_\L$ (or $w_1\circ \rho$).

  First, $d_n=w_1^{2^t-1}$ is clearly in the image of $w_1$ and it is in the image of $\Sq^1$ by Proposition \ref{prop:w1_in_imSq1}. Note that the case $k$ odd and $n=2^t$ in which Proposition~\ref{prop:w1_in_imSq1} (1) doesn't apply is also excluded in point (4) of Conjecture~\ref{conj1}. To show that $d_n$ is not in the image of $w_1\circ \rho_\L$, note that for any $z'$ such that $w_1z'=w_1^{2^t-1}$, we have $z'-w_1^{2^t-2}\in \ker w_1$. Assuming $\deg d_n<\deg a_n$, by  \eqref{eq:crk}, $\ker w_1$ is just $(0)$ in degree $\deg d_n-1$, so $z'=w_1^{2^t-2}$. Therefore it is enough to show that $w_1^{2^t-2}$ is not in the image of $\rho_\L$. And for this it is enough to show that $\Sq^1_\L(w_1^{2^t-2})\neq 0$:
  \[
  \Sq^1_\L(w_1^{2^t-2})=\Sq^1(w_1^{2^t-2})+w_1^{2^t-1},
  \]
  where the first term is zero by the derivation property, and the second term is nonzero by Stong's theorem. 
	
  The case when $\deg a_n<\deg d_n$ proceeds similarly. Note that this implies that we are in case~(3) of Conjecture~\ref{conj1}, i.e., $k$ is even. First, $a_n=w_1^{2^{t-1}-1}w_k^j$ with $j=n-2^{t-1}$ is in the image of $w_1$. It is also in the image of $\Sq^1$ if $n$ is even and in the image of $\Sq^1_\L$ if $n$ is odd, by Proposition \ref{prop:w1_in_imSq1}.  
  Similarly to the first case, it is enough to show that $w_1^{2^t-2}w_k^j$ is not in the image of $\rho_\L$ if $n$ is even and not in the image of $\rho$ if $n$ is odd. Since $\Sq^1 w_k=w_1w_k$, by the derivation property
  \[\Sq^1(w_1^{2^t-2}w_k^j)=j\cdot w_1^{2^t-1}w_k^j=j\cdot a_n,\]	
  which is nonzero if $n$ is odd and 
  \[\Sq^1_\L(w_1^{2^t-2}w_k^j)=(j+1)\cdot a_n,\]
  which is nonzero if $n$ is even, which allows us to conclude.
\end{proof}

Combining this with the characteristic rank conjecture for $n=2^t-1$ and $n=2^t$, as established in Theorem~\ref{thm:charrank} and Corollary~\ref{cor:charrank2t}, we obtain the main result: 

\begin{theorem}  \label{cor:main}\,%Hack
  \begin{enumerate}
  \item 
    For any $t\geq 4$, there is a nontrivial 4-torsion class in ${\rm H}^{2^t-1}(\OGr_5(2^t-1);\Z)$.
  \item For any $t\geq 4$, there is a nontrivial 4-torsion class in ${\rm H}^{2^t-1}(\OGr_6(2^t);\Z)$.
  \end{enumerate}
\end{theorem}

\begin{remark}
  This provides evidence for the items (3) and (4) of Conjecture~\ref{conj1}.
\end{remark}

\begin{remark}
  \label{rem:experiments}
  The 4-torsion condition in Proposition~\ref{prop:2tor_condition} can be implemented in Sage to check small examples. As a basic sanity check, the Sage code verifies that the condition is satisfied for small cases of Theorem~\ref{cor:main}. We mention some additional experimental data supporting or complementing Conjecture~\ref{conj1}.
  \begin{enumerate}
  \item We checked the condition for $\OGr_4(n)$ up to $n=33$ and found no 4-torsion, supporting point (1) of Conjecture~\ref{conj1}. 
  \item We do not expect that the ascended generators $a_n$ give rise to 4-torsion in cases they are the smallest anomalous classes when $k$ is odd. One indication is given in Proposition~\ref{prop:an_oddk}: the class $a_{2^t}$ is the reduction of a non-torsion class for $n=2^t$, and so should be the pushforwards to $\Gr_k(n)$ for $n>2^t$. Using Sage, we also checked small examples: no 4-torsion appears in $\OGr_5(n)$ for $n=16,17$. 
  \item The 4-torsion classes exhibited by Theorem~\ref{cor:main} (as well as the 4-torsion classes that should more generally exist by Conjecture~\ref{conj1}) appear to be only the tip of the iceberg. Theorem~\ref{cor:main} merely exhibits two infinite families of 4-torsion classes we can show exist, and Conjecture~\ref{conj1} explains where 4-torsion classes related to anomalous generators should be found. 
    
    Besides the 4-torsion classes we show exist in Theorem~\ref{cor:main}, there can be more 4-torsion beyond the one arising from the smallest anomalous generator. For example, we list two such cases with the degrees of nontrivial 4-torsion classes:
    \begin{itemize}
    \item $\OGr_5(15)$: 4-torsion in degrees $15, 19, 23, 28, 32, 36$
    \item $\OGr_6(16)$: 4-torsion in degrees $15, 19, 21, 23, 25, 27, 28, 29, 32, 33, 34, 36, 38, 40, 42, 46$
    \end{itemize}
    Note the Poincar\'e duality pattern. There are also noticeable 4-step patterns which indicate that possibly only few generators might be needed, and most of the 4-torsion classes are Pontryagin-class multiples of other 4-torsion classes. 
    
    Even in the cases when we would expect no 4-torsion from anomalous generators, i.e., $k$ odd and $a_n$ being the smallest anomalous class, there can still be substantial 4-torsion in higher degrees. For example, we list four such cases with the degrees of nontrivial 4-torsion classes.
    \begin{itemize}
    \item $\OGr_5(18)$: 4-torsion in degrees $28,32,32,36,36,40$
    \item $\OGr_5(19)$: 4-torsion in degrees $32,33,36,37,40,41$
    \item $\OGr_7(17)$: 4-torsion in degrees $26,30,34,39,43,47$
    \item $\OGr_7(18)$: 4-torsion in degrees $31,32,33,35,36,37,39,39,40,40,41,41,43,44,45,47,48,49$
    \end{itemize}
    It would be interesting to understand the origin of these 4-torsion classes. 
  \end{enumerate}
\end{remark}

\appendix
\section{A dual basis of monomials}
In this appendix, we provide a slightly different version of the proof of Theorem~\ref{thm:charrank}.
\begin{figure}
  \begin{center}
    \begin{tabular}{|c||c|c||}
      \hline
      degree $2^t-3$&$w_2q_{2^t-5}$ & $q_{2^t-3}$\\ \hline\hline
      $w_4^{2^{t-2}-2}w_5$&$\times$&\begin{tabular}{c}$w_4:11\ldots 10$\\$w_5:00\ldots 01$\end{tabular} \\[1ex] \hline
    \end{tabular}
  \end{center}
  \caption{Monomials that form a dual basis for the $W_2$-linear combinations of $q_{2^t-5},q_{2^t-4},q_{2^t-3}$ in degree $2^t-3$. The columns are indexed by a basis of $W_2$-linear combinations $w^bq_j$ of the $q_j$'s in degree $2^t-3$, the rows are indexed by a chosen dual basis of monomials $w^a$. See Theorem \ref{thm:dualbasis} for further details.}
  \label{fig:table_2t-3}
\end{figure}

\begin{figure}
  \begin{center}
    \begin{tabular}{|c||c|c|c||}
      \hline
      degree $2^t-2$&$w_3q_{2^t-5}$&$w_2q_{2^t-4}$ & $q_{2^t-2}$\\ \hline\hline
      $w_3^2w_4^{2^{t-2}-2}$&\color{red}\begin{tabular}{c}$w_3:0\ldots 001$\\$w_4:1\ldots 110$\end{tabular}  &$\times$&\begin{tabular}{c}$w_3:0\ldots 010$\\$w_4:1\ldots 110$\end{tabular} \\[1ex] \hline
      $w_2^{2^{t-1}-3}w_4$&$\times$ &\color{red}\begin{tabular}{c}$w_2:11\ldots 100$\\$w_4:00\ldots 001$\end{tabular} &\begin{tabular}{c}$w_2:11\ldots 101$\\$w_4:00\ldots 001$\end{tabular} \\[1ex] \hline
      $w_4^{2^{t-2}-3}w_5^2$&$\times$ &$\times$&\color{red}\begin{tabular}{c}$w_4:11\ldots 101$\\$w_5:00\ldots 010$\end{tabular} \\[1ex] \hline
    \end{tabular}
  \end{center}
  \caption{Monomials that form a dual basis for the $W_2$-linear combinations of $q_{2^t-5},q_{2^t-4},q_{2^t-3},q_{2^t-2}$ in degree $2^t-2$. The columns are indexed by a basis of $W_2$-linear combinations $w^bq_j$ of the $q_j$'s in degree $2^t-3$, the rows are indexed by a chosen dual basis of monomials $w^a$. See Theorem \ref{thm:dualbasis} for further details.}
  \label{fig:table_2t-2}
\end{figure}

\begin{figure}
  \begin{center}
    \begin{tabular}{|c||c|c|c|c|c||}
      \hline
      degree $2^t-1$&$w_2q_{2^t-3}$&$w_3q_{2^t-4}$ & $w_4q_{2^t-5}$&$w_2^2q_{2^t-5}$&$q_{2^t-1}$\\ \hline\hline
      $w_2^{2^{t-1}-3}w_5$&\color{red}\begin{tabular}{c}$w_2:11\ldots 100$\\$w_5:00\ldots 001$\end{tabular} &$\times$&$\times$&\begin{tabular}{c}$w_2:1\ldots 1011$\\$w_5:0\ldots 0001$\end{tabular}& \begin{tabular}{c}$w_2:11\ldots 101$\\$w_5:00\ldots 001$\end{tabular}\\[1ex] \hline
      $w_2^{2^{t-1}-7}w_3w_5^2$&\color{red} \begin{tabular}{c}$w_2:1\ldots 1000$\\
	$w_3:0\ldots 0001$\\
	$w_5:0\ldots 0010$\end{tabular}&\color{red} \begin{tabular}{c}$w_2:1\ldots 1001$\\
	$w_3:0\ldots 0000$\\
	$w_5:0\ldots 0010$\end{tabular}&$\times$ & \begin{tabular}{c}$w_2:1\ldots 10111$\\
	$w_3:0\ldots 00001$\\
	$w_5:0\ldots 00010$\end{tabular}& \begin{tabular}{c}$w_2:1\ldots 1001$\\
	$w_3:0\ldots 0001$\\
	$w_5:0\ldots 0010$\end{tabular}\\[1ex] \hline
      $w_3w_4^{2^{t-2}-1}$&$\times$ & \color{red}\begin{tabular}{c}$w_3:0\ldots 000$\\$w_4:1\ldots 111$\end{tabular}&\color{red}\begin{tabular}{c}$w_3:0\ldots 001$\\$w_4:1\ldots 110$\end{tabular} & $\times$& \begin{tabular}{c}$w_3:0\ldots 001$\\$w_4:1\ldots 111$\end{tabular}\\[1ex] \hline
      $w_2^2w_3w_4^{2^{t-2}-2}$&\begin{tabular}{c}$w_2:0\ldots 001$\\
	$w_3:0\ldots 001$\\
	$w_4:1\ldots 110$\end{tabular}  &\begin{tabular}{c}$w_2:0\ldots 010$\\
	$w_3:0\ldots 000$\\
	$w_4:1\ldots 110$\end{tabular}  & \begin{tabular}{c}$w_2:0\ldots 010$\\
	$w_3:0\ldots 001$\\
	$w_4:1\ldots 101$\end{tabular} &\color{red} \begin{tabular}{c}$w_2:0\ldots 000$\\
	$w_3:0\ldots 001$\\
	$w_4:1\ldots 110$\end{tabular} &\begin{tabular}{c}$w_2:0\ldots 010$\\
	$w_3:0\ldots 001$\\
	$w_4:1\ldots 110$\end{tabular} \\[1ex] \hline
      $w_2^{2^{t-1}-2}w_3$&\begin{tabular}{c}$w_2:1\ldots 101$\\$w_3:0\ldots 001$\end{tabular} &\color{red}\begin{tabular}{c}$w_2:1\ldots 110$\\$w_3:0\ldots 000$\end{tabular}&$\times$&\color{red}\begin{tabular}{c}$w_2:1\ldots 100$\\$w_3:0\ldots 001$\end{tabular}&\color{red} \begin{tabular}{c}$w_2:1\ldots 110$\\$w_3:0\ldots 001$\end{tabular}\\[1ex] \hline
    \end{tabular}
  \end{center}
  \caption{Monomials that form a separating basis for the $W_2$-linear combinations of $q_{2^t-5},q_{2^t-4},q_{2^t-3},q_{2^t-2},q_{2^t-1}$ in degree $2^t-1$. The columns are indexed by a basis of $W_2$-linear combinations $w^bq_j$ of the $q_j$'s in degree $2^t-3$, the rows are indexed by a chosen dual basis of monomials $w^a$.  See Theorem \ref{thm:dualbasis} for further details.}
  \label{fig:table_2t-1}
\end{figure}

\begin{theorem}\label{thm:dualbasis}
  Let $n=2^t-1$. Then 
  \begin{equation}
    \crk(S\to\OGr_5(n))=2^t-2.
  \end{equation}
\end{theorem}

\begin{proof}
  The characteristic rank is $\leq 2^t-2$ by Theorem \ref{thm:stong-upper-bound}.
  
  To prove that the characteristic rank is $\geq 2^t-2$, we need to show that there are no relations between $q_{2^t-5},\dots, q_{2^t-1}$ in degrees $\leq 2^t-1$. 
  Again, by Korba\v s' lemma in \cite{Korbas2015}, none of these elements are zero, so there are no relations in degrees $2^t-5$ and $2^t-4$. In each of the remaining degrees $2^t-3,2^t-2,2^t-1$ we will give a ``separating basis of monomials'' $w^a$ in the sense that the coefficients of $w^a$ in $w^bq_j$ form a lower triangular matrix.  The existence of such a separating basis implies that any linear relation between $w^bq_j$ has to be trivial. 
  
  A collection of such monomials is given in the left-most column of tables \ref{fig:table_2t-3}, \ref{fig:table_2t-2} and \ref{fig:table_2t-1}. Each cell of these tables describes the coefficient of a monomial $w^a$ (indexing the rows) in $w^bq_j$ (indexing the columns) as follows. 
  First note that the coefficient of $w^a$ in $w^bq_j$ is equal to the coefficient of $w^{a-b}=\prod_{i=2}^5w_i^{a_i-b_i}$ in $q_j$. By the multinomial version of Lucas' theorem (see e.g.\ \cite[Lemma 4.1]{OGr3}), the coefficient of $w^c$ in $q_j$ is nonzero if and only if the binary expansions of $c_2\stb c_5$ are disjoint, i.e.\  there are no 1's in the same position. The tables collect the binary expansions of the exponents $a-b$, and the entries where the binary expansions are disjoint are colored red. The symbol $\times$ means that the monomial $w^a$ appears in $w^bq_j$ with coefficient 0, since $w^b$ does not divide $w^a$. We illustrate the computation on the first row of Table \ref{fig:table_2t-2} -- the computation in all other cases is straightforward and entirely analogous.
  
  We compute the coefficients of the monomial $w_3^2w_4^{2^{t-2}-2}$ in $w_3q_{2^t-5}, w_2q_{2^t-4}$ and $q_{2^t-2}$. In the previous notation, $a_2=0,a_3=2, a_4=2^{t-2}-2, a_5=0$, which have binary expansions $[a_3]=0\ldots 010$ and $[a_4]=1\ldots 110$. Since these are not disjoint, the coefficient of $w^a$ in $q_{2^t-2}$ is zero. Since $w^a$ is not divisible by $w_2$, the coefficient of $w^a$ in $w_2q_{2^t-4}$ is also zero, this is denoted by $\times$ in the first row of Figure \ref{fig:table_2t-2}. Finally, the coefficient of $w^a=w_3^2w_4^{2^{t-2}-2}$ in $w^bq_j=w_3q_{2^t-5}$ is equal to the coefficient of $w^{b-a}=w_3w_4^{2^{t-2}-2}$ in $q_j$. This is nonzero iff the binary expansions of $a_2-b_2\stb a_5-b_5$ are disjoint; and these are $[a_3-b_3]=0\ldots001$ and $[a_4-b_4]=1\ldots 110$, which are disjoint, so the corresponding entry is colored red.
\end{proof}

\end{document}